\documentclass[review]{elsarticle}
\usepackage{lineno,hyperref}

\usepackage{graphicx}
\usepackage[all]{xy}
\usepackage{bm}
\usepackage{amscd,amssymb,amsbsy,amsmath,amsfonts}
\usepackage{color}
\usepackage{fancybox}
\usepackage{enumitem} 
\usepackage{textcomp}

 \usepackage{graphics}%,epstopdf}
\newtheorem{theorem}{Theorem}

\newtheorem{corollary}{Corollary}

\newtheorem{definition}{Definition}

\newtheorem{lemma}{Lemma}

\newtheorem{proposition}{Proposition}
\newtheorem{remark}{Remark}

\usepackage{verbatim}
\modulolinenumbers[5]
\newenvironment{proof}{\paragraph{\textbf{\textrm{Proof}}}}{\hfill$\blacksquare$}
\journal{Journal of Symbolic Computation}

\def\ad{\eta}
\def\bd{\zeta}

\begin{document}

\begin{frontmatter}

\title{Liouvillian solutions for second order linear differential equations with Laurent polynomial coefficient}

\author{Primitivo B. Acosta-Hum\'anez}
\fnref{myfootnote}
\address{Instituto de Matem\'atica, Facultad de Ciencias\\\href{www.uasd.edu.do}{Universidad Aut\'onoma de Santo Domingo}\\ Santo Domingo, Dominican Republic.}
\ead{pacosta-humanez@uasd.edu.do}

%\address{School of Basic and Biomedical Sciences,  \href{www.unisimonbolivar.edu.co}{Universidad Sim\'on Bol\'{\i}var}\\  Barranquilla, Colombia.}
\fntext[myfootnote]{Corresponding author}

\author{David Bl\'azquez-Sanz} 
\address{Escuela de Matem\'aticas, Facultad de Ciencias\\ \href{https://unal.edu.co/}{Universidad Nacional de Colombia} - Sede Medell\'in\\ Medell\'in - Colombia.}
\ead{dblazquezs@unal.edu.co }

\author{Henock Venegas-G\'omez}
\address{Escuela de Matem\'aticas, Facultad de Ciencias\\ \href{https://unal.edu.co/}{Universidad Nacional de Colombia} - Sede Medell\'in\\ Medell\'in - Colombia.\\ Escuela de Ciencias B\'asicas y Tecnolog\'ia e Ingenier\'ia, \href{www.unad.edu.co}{Universidad Nacional Abierta y a Distancia},  Medell\'in, Colombia.}
\ead{hvenegasg@unal.edu.co}

\begin{abstract}
This paper is devoted to a complete parametric study of Liouvillian solutions of  the general trace-free second order differential equation with a Laurent polynomial coefficient. This family of equations, for fixed orders at $0$ and $\infty$ of the Laurent polynomial, is seen as an affine algebraic variety. We proof that the set of Picard-Vessiot integrable differential equations in the family is an enumerable union of algebraic subvarieties. We compute explicitly the algebraic equations of its components. We give some applications to well known subfamilies as the doubly confluent and biconfluent Heun equations, and to the theory of algebraically solvable potentials of Shr\"odinger equations. Also, as an auxiliary tool, we improve a previously known criterium for second order linear differential equations to admit a polynomial solution. 
\end{abstract}


\end{frontmatter}

%\linenumbers

\section{Introduction}

%The obtaining of explicit solutions of large families of differential equations is a topic in which a plenty of mathematicians work since some centuries ago. The importance of such solutions include new results in other areas such as mathematical physics, bio-mathematics, among others. One typical example is the raising and lowering method to solve Schr\"odinger equations. Such method, also known as algebraic method, involve the explicit solution for the initial value of the energy. In this sense, numerical methods are not useful. Moreover, explicit solutions can involve parameters and they do not need initial conditions, while numerical methods need must be applied with initial conditions and without parameters.

%There is a long standing interest in the computation of families of explicit solutions of differential equations. 
There is no general agreement about the meaning of the term ``explicit solution'' for ordinary differential equations. Some authors allow certain special functions, some others consider only elementary functions. This paper is written inside the framework of the differential Galois theory of linear differential equations, also known as Picard-Vessiot theory. Therefore, our notion of explicit solution is that of Liouvillian function. In the particular case of linear differential equations with rational coefficients, solutions that can be expressed in terms of elementary functions and indefinite integrals (integration by quadratures) are always Liouvillian functions. It is also well known that a second order linear differential equation with rational coefficients admits a Liovillian solution if and only it is Picard-Vessiot integrable. Previous theoretical results on Picard-Vessiot theory by M. F. Singer \cite{singer1993moduli} ensure that given a finite dimensional family of linear differential equations, the subfamily of Picard-Vessiot integrable equations admits a canonical description as union of algebraic subvarities. Our objective is to give an explicit description of the subfamily of integrable Picard-Vessiot equations inside the family of second order trace-free linear differential equations with Laurent polynomial coefficient, and the Liouvillian solutions attached to those integrable equations. Picard-Vessiot integrability of equations of some subfamilies of the families studied here, corresponding to confluences of hypergeometric and Heun equations, where already studied in detail by A. Duval and Loday-Richaud in \cite{dulo92}. \\

We consider the family of second order trace-free linear differential equations, 
\begin{equation}\label{eq:general}
\frac{d^2 y}{dx^2} = L(x)y, \quad L(x)\in \mathbb C[x,x^{-1}]
\end{equation}
where $L(x)\in \mathbb C[x,x^{-1}]$ is a monic Laurent polynomial
$$L(x) = \frac{\ell_{-r}}{x^r} + \ldots + \ell_0 + \ldots  + \ell_{m-1}x^{m-1} + x^m$$
with $\ell_{-r}\neq 0$. We say that $L(x)$ has type $(r,m)$. The space $\mathbb M_{(r,m)}$ of monic Laurent polynomials of type $(r,m)$ is an affine algebraic variety $\mathbb M_{(r,m)} \simeq \mathbb C^*\times \mathbb C^{r+m-1}$. The purpose of this paper is to classify Picard-Vessiot integrable equations in the family \eqref{eq:general} and to study how their Liouvillian solutions 
depend algebraically on the coefficients of $L(x)$ when it moves in the space $\mathbb M_{(r,m)}$. We thus introduce the \emph{spectral set}.

\begin{definition}
The \emph{spectral set} $\mathbb S_{(r,m)}\subset \mathbb M_{(r,m)}$ is the set of monic Laurent polynomials of type $(r,m)$ such that its corresponding Eq. \eqref{eq:general} is Picard-Vessiot integrable (or equivalently, has a Liouvillian solution). 
\end{definition}
 
For the formal definition of the Galois group,  Picard-Vessiot integrability, and Kovacic algorithm we refer the readers the original article \cite{Kov86} of J. Kovacic. 
%and the appendix of \cite{AB08}. 
Nevertheless, we included an appendix with the parts of the algorithm that are relevant for the our results.

The Picard-Vessiot integrability analysis of Eq. \eqref{eq:general} is done in two steps. Theoretically, Kovacic algorithm admits three different cases of integrability. However, by means of the D'Alembert transform, we show that, after a two-sheet covering of the Riemann sphere, we can reduce the analysis to the first case.
This reduces the problem to  the existence of a polynomial solution of an attached auxiliary equation. To deal with this existence problem, we give an algebraic generalization (Theorem \ref{ThMiaPoly}) of the asymptotic iteration method (AIM) due to Cifti, Hall and Saad \cite{CHS}. Here we discover a universal family of differential polynomials in two variables (Table \ref{useries}) that controls the existence of polynomial solutions for second order differential equations with coefficients in arbitrary differential fields.  \\

We arrive to a decomposition of the spectral sets as enumerable union of spectral varieties (Theorems \ref{th:case1}, \ref{th:case2}, \ref{th:case3} and Propositions \ref{th:aim1}, \ref{th:aim2}, \ref{th:aim3}) whose equations can be given explicitly by means of universal polynomials $\Delta_d$ and auxiliary equations (\ref{aux1}, \ref{aux2}, \ref{aux3}). These results are summarized in the following: \\

\noindent{\bf Theorem \ref{th:main}} \emph{
For any $(r,m)\in \mathbb Z_{>0}^2$ we have a decomposition of the spectral set
$${\mathbb S}_{(r,m)} = \bigcup_{d\geq 0} {\mathbb S}_{(r,m)}^{(d)}$$
as a enumerable union of spectral varieties. Moreover:
\begin{enumerate}
    \item[(a)] If $L(x)\in \mathbb S_{(r,m)}^{(d)}$ with $r\neq 2$ then the differential equation \eqref{eq:general} has a solution of the form:
$$y(x) = x^{\lambda}P(x)e^{\int \omega(x)dx}$$
where $P(x)$ is a monic polynomial of degree $d$, and $\omega(x)$ is a Laurent polynomial. 
   \item[(b)] If $L(x)\in \mathbb S_{(r,m)}^{(d)}$ with $r=2$ then the differential equation \eqref{eq:general} has a solution of the form:
$$y(x) = x^{\lambda}P(\sqrt{x})e^{\int \omega(\sqrt{x})dx}$$
where $P(x)$ is a monic polynomial of degree $d$, and $\omega(x)$ is a Laurent polynomial.
\end{enumerate}
} 

%Then we give some applications and examples, explainig or expanding some other resuts. {\color{blue} Recently, Saad, Hall and Cifty, see \cite{SHC}, obtained the polynomial solutions for some special functions (Laguerre, Legendre, etc.).} 
In section \ref{s:examples} present some applications that involve the analysis of biconfluent Heun equation and doubly confluent Heun equation, 
as well Schr\"odinger Equations with Mie potentials (Laurent polynomial potentials, exponential potentials) and Inverse Square Root potentials, see \cite{AAT18,CY19,saad2020}. Finally, our results allow us to state that there are no new algebraically solvable Laurent polynomial potentials for the Shcr\"odinger equation beyond those previously known (Corolary \ref{cor1}) corresponding to $m=2$ and $r = 0,1,2$. 

\subsection{A note on the general case}
The assumptions of having a trace-free differential equation with a monic Laurent polynomial is done with the purpose of presenting clearer computations and formulae. In fact, our analysis does applies to the more general case of differential equations of the form,
\begin{equation}\label{eq:muygeneral}
\frac{d^2\tilde y}{d\tilde x^2} = L_0(x)\tilde y + L_1(x)\frac{d\tilde y}{dx}, \quad L_0(\tilde x),\, L_1(\tilde x) \in \mathbb C[\tilde x,\tilde x^{-1}].
\end{equation}
\begin{itemize}
\item First, we can always reduce Eq. \eqref{eq:muygeneral} to trace-free from by means of the so called D'alembert transform. Namely, we look for a suitable scaling  $y = f(x)\tilde y$ of the dependent variable, so that the second derivative of the new unknown function is:
$$\frac{d^2y}{d\tilde x^2} = \left(\frac{f''(\tilde x)}{f(\tilde x)} + L_0(\tilde x)\right)y + (2f'(\tilde x) + f(\tilde x)L_1(\tilde x))\frac{d\tilde y}{d\tilde  x}.$$
Here, the term in $\frac{d\tilde y}{d \tilde x}$ vanish if we take $f$ to be a solution of the differential equation $f' = \frac{1}{2} L_1(\tilde x)f$ wich easily yields,
$$f(\tilde x) = {\rm exp}\left( - \frac{1}{2}\int L_1( \tilde x) d\tilde x \right)$$
and a trace-free differential equation,
$$\frac{d^2y}{d\tilde  x} = \left(\frac{L_1(\tilde x)^2}{4}  - \frac{1}{2}\frac{dL_1(\tilde x)}{d\tilde x} + L_0(\tilde x) \right)y,$$
depending of a Laurent polynomial whose coefficients are polynomials of degree $1$ and $2$ in the coefficients of $L_0(\tilde x)$ and $L_1(\tilde  x)$.
\item Second, the trace-free differential equation above with non necessarily monic Laurent polynomial, 
$$\left(\frac{L_1(\tilde x)^2}{4}  - \frac{1}{2}\frac{dL_1(\tilde x)}{d\tilde x} + L_0(\tilde x) \right) = \frac{a_{-r}}{\tilde x^r} + \ldots + a_m \tilde x^m, \quad \ell_r \neq 0,$$
can be transformed into Eq. \eqref{eq:general2} by scaling the independent variable $x = \sqrt[m+2]{a_m} \tilde  x$. We fall in the monic case, and the coefficients of the new Laurent polynomial are polynomials in $a_{-r}$, $\ldots$, $a_{m-1}$, $\sqrt[m+2]{a_m}$.
\item Third, the assumption of $r \cdot m > 0$ is not necessary. If $r = 0$ then we fall in the simpler polynomial case that has been studied exhaustively by the authors in \cite{abv2020} as a previous step for the analysis of the Laurent polynomial case. If $m = 0$ then we may just apply consider $z = \frac{1}{x}$ as a new independent variable and fall into the $m\cdot r >0$ case. 
\end{itemize}

\section{Kovacic algorithm analysis}

%In order to understand the results and proofs in this paper it is required some familiarity with the Kovacic algorithm \cite{Kov86} for determining Picard-Vessiot integrability of a second order linear differential equations with rational coefficients. 
%The most relevant fact is that 
The most relevant fact about Kovacic algorithm is that the Galois group of Eq. \eqref{eq:general} with fixed $L(x)$ is an algebraic subgroup of ${\rm SL}_2(\mathbb C)$, and the equation is integrable if and only if such group is conjugated to a subgroup of the Borel subgroup 
$${\rm B}_2 = \left\{ \left[\begin{array}{cc} \lambda & \mu \\ 0 & \lambda^{-1} \end{array}\right] \,\colon \, \lambda\in \mathbb C^*,\,\mu \in \mathbb C \right\},$$
the dihedral group 
$${\rm D}^{\infty} = \left\{ \left[\begin{array}{cc}\lambda & 0\\ 0 & \lambda^{-1} \end{array}\right] \,\colon \, \lambda\in \mathbb C^*\right\} \cup \left\{ \left[\begin{array}{cc} 0 & -\lambda \\ \lambda^{-1} & 0 \end{array}\right] \,\colon \, \lambda\in \mathbb C^*\right\},$$
or it is finite and not contained in the former. These correspond to the so-called first three cases of Kovacic algorithm \cite{Kov86}, the fourth being the non-integrable case. In what follows, for practical purposes related with the aplication of the algorithm, we split the set $\mathbb Z_{>0}^2$ in four different disjoint subsets, corresponding to different orders of Laurent polynomials at $0$ and $\infty$ that allow different possibilities for Liouvillian solutions of Eq. \eqref{eq:general}.

\begin{definition}
We say that a pair $(r,m)\in \mathbb Z^{2}_{>0}$ is of class:
\begin{itemize}
    \item[(1)] $r\in\{1\}\cup\{2k+4\colon k\in\mathbb Z,\,k\geq 0 \}$ and $m$ is even;
    \item[(2)] $r=2$ and $m$ is odd;
    \item[(3)] $r=2$ and $m$ is even;
    \item[(4)] not in classes (1), (2) or (3).
\end{itemize}
So that, we have a disjoint partition $\mathbb Z^2_{>0} = \mathcal C_1 \cup \mathcal C_2 \cup \mathcal C_3 \cup \mathcal C_4$. 
\end{definition}

\begin{lemma}\label{lm:decomp}
Let us assume that Eq. \eqref{eq:general} with some fixed $L(x) \in \mathbb M_{(r,m)}$ is Picard-Vessiot integrable. The following statements hold.
\begin{itemize}
    \item[(a)] If $(r,m)\in \mathcal C_1$ then the Galois group is conjugated to a subgroup of ${\rm B_2}$.
    \item[(b)] If $(r,m)\in \mathcal C_2$ then the Galois group is conjugated to a subgroup of ${\rm D^{\infty}}$.
    \item[(c)] If $(r,m)\in \mathcal C_3$ then the Galois group is conjugated to a subgroup of ${\rm B}_2$ or ${\rm D^{\infty}}$
    \item[(d)] $(r,m)\not\in \mathcal C_4$.
\end{itemize}
\end{lemma}

\begin{proof}
  This lemma is a consequence of the necessary conditions in Kovacic algorithm (Theorem \ref{ap:NC} in \ref{ap:kov}). The order of $L(x)$ at $x=\infty$ is $-m$, a negative integer. Thus, case $3$ of Kovacic algorithm is discarded. Thus, the remaining possibilities are: the Galois group is either conjugated to a subgroup of  $\rm B_2$ or to a subgroup of ${\rm D^{\infty}}$. Let us assume that the Galois group is conjugated to a subgroup of $\rm B_2$ then, necessary conditions for case $1$ must be satisfied. Therefore, $L(x)$ is either of type $(1,2p)$ or $(2q,2p)$ with $q>1$, \textit{i.e.}, $(r,m)\in\mathcal{C}_1 \cup \mathcal{C}_3$. 

On the other hand, let us suppose that the Galois group of Eq. \eqref{eq:general} is conjugated to a subgroup of ${\rm D^{\infty}}$, in this case $L(x)$ is of type $(2,m)$ or $(2q+1,m)$ with $q>0$. The first possibility $(r,m) = (2,m)$ clearly corresponds to $(r,m) \in \mathcal{C}_2\cup\mathcal{C}_3$. The second possibility $(r,m) = (2q+1,m)$ with $q>0$ is discarded if step 2 of case 2 of the algorithm, which requires the quantity $d=\frac{1}{2}(-m-2q-1)$ to be a non-negative integer.
\end{proof}

%\begin{proposition}
%Let us fix $L(z)\in\mathbb M_{(r,m)}$ and assume that equation \eqref{eq:general} is Picard-Vessiot integrable. The following statements hold:
%\begin{enumerate}
% \item Its Galois group is conjugated to a subgroup of either ${\rm B}_2$ or ${\rm D}^\infty$.
%\item $(r,m)\not\in \mathcal C_4$. 
%\item If $(r,m)\in \mathcal C_1\cup \mathcal C_3$ then its Galois group is conjugated to a subgroup of ${\rm B_2}$.
%\end{enumerate}
%\end{proposition}

%\textbf{Proof.} \emph{
%{\color{red} Cosas de la tesis }
%\}

Let us introduce the following notation:
\begin{itemize}
\item[(a)] $\mathbb B_{(r,m)}\subset \mathbb S_{(r,m)}$ is the set of Laurent polynomials $L(x)\in \mathbb M_{(r,m)}$ such that the Galois group of Eq. \eqref{eq:general} is conjugated to a subgroup of ${\rm B}_2$.
\item[(b)] $\mathbb D_{(r,m)}\subset \mathbb S_{(r,m)}$ is the set of Laurent polynomials $L(x)\in \mathbb M_{(r,m)}$ such that the Galois group of Eq. \eqref{eq:general} is conjugated to a subgroup of ${\rm D}^{\infty}$.
\end{itemize}

\begin{remark}
From Lemma \ref{lm:decomp} we have:
\begin{itemize}
    \item[(a)] If $(r,m) \in \mathcal C_1$ then $\mathbb S_{(r,m)} = {\mathbb B}_{(r,m)}$
    \item[(b)] If $(r,m) \in \mathcal C_2$ then $\mathbb S_{(r,m)} = {\mathbb D}_{(r,m)}$.
    \item[(c)] If $(r,m) \in \mathcal C_3$ then
    $\mathbb S_{(r,m)} = {\mathbb B}_{(r,m)}\cup{\mathbb D}_{(r,m)}$
    \item[(d)] If $(r,m) \in \mathcal C_4$ then $\mathbb S_{(r,m)} = \emptyset$
\end{itemize}
\end{remark}

\subsection{Characterization of $\mathbb S_{(1,2p)}$.}

Let us consider the case with $r=1$ and $m = 2p$ even. In such case we have a unique decomposition:
\begin{equation}\label{Q_decomp_1}
L(x) = \frac{a}{x} + B(x) + A(x)^2
\end{equation}
where $A(x)$ is a monic polynomial of degree $p$ and quadratic residue $B(x)$ is a polynomial of degree $p-1$. We also have in mind that the map
$$\mathbb M_{(1,m)} = \mathbb C^*\times \mathbb C^{2p} \to \mathbb C^*\times \mathbb C^{2p}, \quad (\ell_{-1},\ell_0,\ldots,\ell_{m-1})\mapsto (a,a_0,\ldots,a_{p-1},b_0,\ldots, b_{p-1})$$
is an invertible polynomial map. 

\begin{theorem}\label{th:case1}
Differential equation \eqref{eq:general} with coefficient $L(x)\in \mathbb M_{(1,2p)}$ as in \eqref{Q_decomp_1} is Picard-Vessiot integrable, that is, $L(x)\in \mathbb S_{(1,2p)}$, if and only for a choice of  sign $s_{\infty} = \pm 1$ the following conditions hold:
  \begin{enumerate}
   \item The quantity $$d = \frac{s_\infty b_{p-1}-p-2}{2}$$ is a non negative integer. 
   \item There exist a polynomial $P(x)$ of degree $d$ such that:
  \begin{equation}\label{aux1} \tag{${\rm A}_1$}
   P''(x)+2\bigg(s_\infty A(x)+\frac{1}{x}\bigg)P'(x)+\bigg(s_\infty A'(x)-B(x)+\frac{2s_\infty A(x)-a}{x}\bigg)P(x)=0.
\end{equation}
   \end{enumerate}
   In such case, the liouvillian solution
   \begin{eqnarray*}
   y=xP(x)e^{\int s_\infty A(x)dx}
   \end{eqnarray*}
   is an eigenvector of the Galois group.
\end{theorem}

\begin{proof}
red We are in case $\mathcal C_1$ and, by Lemma \ref{lm:decomp}, if the equation is Picard-Vessiot integrable then it corresponds to case 1 of Kovacic algorithm.
%Let us consider Eq. \eqref{eq:general} with Laurent polynomial coefficient as stated in \eqref{Q_decomp_1}, this equation possess singularities at $x_0=0$ and $x_1=\infty$ of order $o(L_0)=1$ and $o(L_\infty)=-2p$, respectively. Necessary conditions in Kovacic algorithm establishes, under the conditions specified before the analysis of Picard-Vessiot integrability of this equation lays on case one of algorithm, in particular, 
Step $1$ of case $1$ gives us conditions $\{c_1, \infty_3\}$,
\begin{equation*}
   [\sqrt{L}]_0=0,\hspace{5pt}\alpha_0^{\pm}=1.
\end{equation*}
in addition, regarding to $x_1$, we got that
\begin{equation*}
   [\sqrt{L}]_\infty=A(x),\hspace{5pt}\alpha_\infty^{\pm}=\frac{1}{2}(\pm b_{p-1}-p).
\end{equation*}
For each choice of the sign $s_{\infty} \in \{+1,-1\}$ we consider the complex number
$d(s_\infty)=\frac{1}{2}(s_\infty b_{p-1}-p)-1$. If none of them is a non-negative integer then we can discard case 1 and the Galois group
${\rm SL}_2(\mathbb{C})$. 
%Now, let us consider the set of non-negative integers $D=\{d\in\mathbb{Z}_+|d=\frac{1}{2}(s_\infty b_{p-1}-p)-1\}$. If $D$ is  empty, then we can discard case 1 and the Galois group %of Eq. \eqref{eq:general} is exactly the special linear group 
${\rm SL}_2(\mathbb{C})$. 
Otherwise  we proceed to step 2 of case 1 for each suitable value of $d = d(s_\infty)$. We set the rational function $\omega=s_\infty A(x)+\frac{1}{x}$ and search for a monic polynomial $P$ of degree $d$, which satisfy the auxiliary differential equation,
\begin{equation*}
    P''+2\omega P'+(\omega ' +\omega^2 -L)P=0.
\end{equation*}
This last equation can be written in terms of the decomposition showed in Eq. \eqref{Q_decomp_1}, obtaning Eq. \eqref{aux1}
of the statement.
%\begin{equation*}
%   P''(x)+2\bigg(s_\infty A(x)+\frac{1}{x}\bigg)P'(x)+\bigg(s_\infty A'(x)-B(x)+\frac{2s_\infty A(x)-a}{x}\bigg)P(x)=0.
%\end{equation*}
If a pair $(\omega, P)$ described as above can be found, Eq. \eqref{eq:general} is Picard-Vessiot integrable. Moreover, Kovacic algorithm provides us the solution given in the statement.
%\begin{equation*}
%    y=xP(x)e^{\int s_\infty A(x)dx}.
%\end{equation*}

\end{proof}

\subsection{Characterization of $\mathbb B_{(2,2p)}$.}

Let us consider the case with $r=2$ and $m$ even. We have a unique decomposition:
\begin{equation}\label{Q_decomp_2}
L(x) = \frac{b}{x^2} + \frac{a}{x} + B(x) + A(x)^2
\end{equation}
where $A(x)$ is a monic polynomial of degree $p$ and quadratic residue $B(x)$ is a polynomial of degree $p-1$. As in the above case, the map
$$\mathbb M_{(2,m)} = \mathbb C^*\times \mathbb C^{2p+1} \to \mathbb C^*\times \mathbb C^{2p+1},$$  $$(\ell_{-2},\ell_{-1},\ell_0,\ldots,\ell_{m-1})\mapsto (b,a,a_0,\ldots,a_{p-1},b_0,\ldots, b_{p-1})$$
is an invertible polynomial map. 

\begin{theorem}\label{th:case2}
The Galois group of Eq. \eqref{eq:general} with fixed $L(x)\in\mathbb B_{(2,2p)}$
is conjugated to a subgroup of ${\rm B}_2$, that is, $L(x)\in\mathbb B_{(2,2p)}$, if and only if for a combination of  signs $s_\infty=\pm 1$ and $s_0 = \pm 1$ the following conditions hold:
\begin{enumerate}
	\item  The quatity 
	$$d = \frac{s_\infty b_{p-1}-s_0 \sqrt{1+4b}-p-1}{2}$$ 
	is a non negative integer.
	\item There exist a polynomial $P(x)$ of degree $d$ such that:
	\begin{equation}\label{aux2} \tag{${\rm A}_2$}
P''(x)+2\bigg(s_\infty A(x)+\frac{\lambda }{x}\bigg)P'(x)+ \bigg(s_\infty A'(x)-B(x)+\frac{2 s_\infty \lambda A(x) -a}{x}\bigg)P(x)=0.
	\end{equation}
    with 
    $$ \lambda = \frac{1 + s_0\sqrt{1+4b}}{2} = \frac{s_{\infty}b_{p-1}-p-2d}{2}.$$
	\end{enumerate}
 In such case, the liouvillian solution 
\begin{equation*}
y=x^\lambda P(x)e^{\int s_\infty A(x)dx}
\end{equation*}	
is an eigenvector of the Galois group. 
\end{theorem}

\begin{proof}

This proof is similar to that of Theorem \ref{th:case1}. By Lemma \ref{lm:decomp} we are in case $(r,m)\in\mathcal C_2$.
We proceed to step 1 of case 1 of Kovacic  algorithm with conditions $\{c_2,\infty_3\}$, 
\begin{equation*}
   [\sqrt{L}]_0=0,\hspace{5pt}\alpha_0^{\pm}=\frac{1}{2}(1\pm\sqrt{1+4b}),
\end{equation*}
\begin{equation*}
   [\sqrt{L}]_\infty=A(x),\hspace{5pt}\alpha_\infty^{\pm}=\frac{1}{2}(\pm b_{p-1}-p).
\end{equation*}

For each choice of $s_0 = \pm 1$ and $s_{\infty} = \pm 1$ we consider the complex number:
$$d(s_0,s_{\infty} ) = \frac{1}{2}(s_\infty b_{p-1}-p)-\frac{1}{2}(1+s_0\sqrt{1+4b}).$$
If none of them is a non-negative integer then we discard case $1$. Otherwise,
for each non-negative integer value of $d = d(s_0,s_\infty)$ we  
set the rational function
$$\omega=s_\infty A(x)+\frac{1}{2}\frac{1+s_0\sqrt{1+4b}}{x}=s_\infty A(x)+\frac{\lambda}{x}$$ 
and
proceed to the step 2 of case 1 of Kovacic algorithm. We search for a monic polynomial $P$ of degree $d$, which satisfy the auxiliary differential equation,
\begin{equation*}
    P''+2\omega P'+(\omega ' +\omega^2 -L)P=0.
\end{equation*}
And, using Eq. \eqref{Q_decomp_2} we arrive to the expression of Eq. \eqref{aux2}. 
%Auxiliary equation above can be expressed using the decomposition given in Eq. \eqref{Q_decomp_2}:
%\begin{equation*}
%P''(x)+2\bigg(s_\infty A(x)+\frac{\lambda }{x}\bigg)P'(x)+ \bigg(s_\infty A'(x)-B(x)+\frac{2 s_\infty \lambda A(x) %-a}{x}\bigg)P(x)=0.
%	\end{equation*}
If a pair $(\omega, P)$ described as above can be found, Eq. \eqref{eq:general} is Picard-Vessiot integrable. In addition, Kovacic algorithm  provides us the solution given in the statement.
%\begin{equation*}
%    y=x^\lambda P(x)e^{\int s_\infty A(x)dx}.
%\end{equation*}
\end{proof}

\subsection{Characterization of $\mathbb S_{(2q,2p)}$.}

For $Q(x)\in \mathbb M_{(2q,2p)}$ with $q>1$ we look for a decomposition that takes into account the cuadratic residues at zero and infinity simultaneously. 
\begin{equation}\label{Q_decomp_3}
L(x) = R(x)^2 + B(x) + A(x)^2, 
\end{equation}
with,
$$R(x) = \frac{r_{-q}}{x^q} + \ldots +\frac{r_{-2}}{x^2},$$
$$B(x) = \frac{b_{-(q+1)}}{x^{q+1}}+ \ldots + b_{p-1}x^{p-1},$$
$$A(x) = a_0 + \ldots + a_{p-1}x^{p-1} + x^p.$$
In this case $R(x)$ is not uniquely determined, as both $R(x)$ and its reciprocal $-R(x)$ can be used in the decomposition. The coefficients of $R(x)$ are polynomials in some of the the coefficients of $L(x)$ and $r_{-q} = \sqrt{\ell_{-2q}}$. Coefficients of $B(x)$ and $A(x)$ are polynomials in some of the coefficients of $L(x)$. What we have is that the map,
$$\pi_{2q,2p}\colon \tilde{\mathbb M}_{(2q,2p)} = \mathbb C^* \times \mathbb C^{2(p+q)-1} \to \mathbb M_{(2q,2p)}$$
$$(r_{-q},\ldots,r_{-2},b_{-(q+1)},\ldots,b_{p-1},a_0,\ldots,a_{p-1}) \to L(x)$$
is a two-sheet cover, with the advantage that elements $(R(x),B(x),A(x))\in \tilde{\mathbb M}_{(2q,2p)}$ correspond to specific decomposition. Let us denote by $\tilde{\mathbb S}_{(2q,2p)} = \pi_{(2q,2p)}^{-1}(\mathbb S_{(2q,2p)})$ the pullback of the spectral set. 

From now, let us fix a decomposition as in Eq. \eqref{Q_decomp_3} and work with the differential equation,
\begin{equation}\label{eq:general2}
y'' = (R(x)^2 + B(x) + A(x)^2)y
\end{equation}
For each choice of $s_0 = \pm 1$ and $s_{\infty} = \pm 1$ let us define the following quantities and functions:
\begin{equation}\label{eq:d_y_lambda}
\begin{split}
d &=  \frac{s_\infty b_{p-1} -p - q}{2} - \frac{s_0b_{-(q+1)}}{2r_{-q}},\\
\lambda &= \frac{s_0b_{-(q+1)}}{2r_{-q}} + \frac{q}{2} = \frac{s_\infty b_{p-1} - p}{2} - d, \\
\omega(x) &= s_\infty A(x) + s_0R(x).
\end{split}
\end{equation}

Note that changing the sign of $s_0$ is equivalent to changing the choice of $R(x)$ in the decomposition. 

\begin{theorem}\label{th:case3}
The differential equation \eqref{eq:general2} with coefficient is Picard-Vessiot integrable,  that is, $(R(x),B(x),A(x))\in \tilde{\mathbb S}_{(2q,2p)}$, if and only if for a combination of signs $s_0 = \pm 1$ and $s_{\infty} = \pm 1$ the following conditions hold:
\begin{enumerate}
\item $d$ is a non negative integer.
\item There exist a polynomial $P(x)$ of degree $d$ such that:
\begin{equation}\label{aux3} \tag{${\rm A}_3$}
P''(x) + 2\left( \omega(x) + \frac{\lambda}{x}\right) P'(x) + \left(\omega'(x) - B(x) + \frac{\lambda(2\omega(x) + \lambda-1)}{x}\right)P(x) = 0. 
\end{equation}
\end{enumerate}
In such case, the liouvillian solution 
\begin{eqnarray*}
y=x^{\lambda}P(x)e^{\int \omega(x)dx}
\end{eqnarray*}
is an eigenvector of the Galois group.
\end{theorem}

\begin{proof}
%Let us consider the differential equation \eqref{eq:general2}. In order to check Picard-Vessiot integrability notice that Eq. \eqref{eq:general2} have singularities at $x=0$ and $x=\infty$ of order $o(L_0)=2q$ and $o(L_\infty)=-2p$, necessary conditions in Kovacic algorithm implies this analysis must be realize under case one of the algorithm, especially in condition 

By Lemma \ref{lm:decomp}, if the Eq. \eqref{eq:general2} is Picard-Vessiot integrable then it corresponds to case $1$ of Kovacic algorithm. We proceed to step 1 of case $1$ obtaining a case of type 
$\{c_3,\infty_3\}$.
\begin{equation*}
[\sqrt{L}]_0=R(x),\hspace{5pt}\alpha_0^{\pm}=\frac{1}{2}\left(\pm\frac{b_{-(q+1)}}{r_{-q}}+q\right).
\end{equation*}
%It also defines, for the singularity at $x=\infty$, that
\begin{equation*}
[\sqrt{L}]_{\infty}=A(x),\hspace{5pt}\alpha_\infty^{\pm}=\frac{1}{2}\left(\pm b_{p-1}-p\right).
\end{equation*}
A necessary integrability condition, given in step 2 of case 1, is that some of the quantities $d$ (as defined in Eq. \eqref{eq:d_y_lambda}), corresponding to some
choice of signs $s_0 = \pm 1$ and $s_\infty = \pm 1$, is a non-negative integer. If not, the Galois group of Eq. \eqref{eq:general2} is ${\rm SL}_2(\mathbb{C})$.
%Let us consider the set of non-negative integers $D=\{d\in\mathbb{Z}_+|d=\frac{1}{2}(s_\infty b_{p-1}-p)-\frac{1}{2}(s_0\frac{b_{-(q+1)}}{r_{-q}}+q)\}$. If $D$ is an empty set, immediately the Galois group of Eq. \eqref{eq:general2} is exactly the special linear group ${\rm SL}_2(\mathbb{C})$. 
Otherwise, for each suitable choice of signs we set the rational function $\phi=s_\infty A(x)+s_0 R(x)+\frac{1}{2}\frac{s_0\frac{b_{-(q+1)}}{r_{-q}}+q}{x}=\omega(x)+\frac{\lambda}{x}$ and search for a monic polynomial $P$ of degree $d$, which satisfy the auxiliary differential equation,
\begin{equation*}
    P''+2\phi P'+(\phi ' +\phi^2 -L)P=0.
\end{equation*}
This equation assumes the form of Eq. \eqref{aux3} when we substitute the decomposition given in \eqref{Q_decomp_3}.
%\begin{equation*}
%P''(x) +2 \left( \omega(x) + \frac{\lambda}{x}\right) P'(x) + \left(\omega'(x) - B(x) + \frac{\lambda(2\omega(x) + %\lambda-1)}{x}\right)P(x) = 0. 
%	\end{equation*}
If a pair $(\omega, P)$ described as above can be found then Eq. \eqref{eq:general2} is Picard-Vessiot integrable. In addition, Kovacic algorithm  provides us the solution given in the statement.
%\begin{equation*}
 %   y=x^\lambda P(x)e^{\int s_\infty A(x)dx}.
  %  \end{equation*}
\end{proof}

\subsection{Characterization of $\mathbb S_{(2,m)}$ in terms of ${\mathbb B}_{(2,2m+2)}$}

Let us take differential equation \eqref{eq:general} with $L(x)\in \mathbb S_{(2,m)}$. It implies that the associated Riccati equation,
$$u = \frac{y'}{y}, \quad u' = u^2 + L(x)$$
has an algebraic solution $u(x)$ of degree $1$ (case $\mathbb B_{(2,m)}$) or $2$ (case $\mathbb D_{(2,m)}$). 
Let us assume that $L(x)\in \mathbb D_{(2,m)}$.  Now we have a look on the geometric properties of the solutions of the Riccati equation. It is well known (see \cite{de2016uniformization} Proposition VIII.1.1) that given an initial condition $(x_0,u_0)\in \mathbb C^* \times \bar{\mathbb C}$ and any simply connected open subset $\mathcal U$ of $\mathbb C^*$ there is a solution $u$ with $u(x_0) = u_0$ and meromorphic in $\mathcal U$. This implies that the solutions of the Riccati equation may have poles but not ramification points outside of the singular locus $\{0, \infty\}$. 
If we compose the algebraic solution with the ramified cover,
$$\bar{\mathbb C}\to \bar{\mathbb C}, \quad w \mapsto x = w^2$$
we obtain that the two determinations of the algebraic function $u$ split in two rational functions of $w$ with poles only in $0$ and $\infty$. In other words, the Riccati equation has a solution which is a rational function on $\sqrt{x}$.

Therefore, if we apply the change of variable $x = w^2$ %, and thus:
%$$\frac{d^2y}{dw^2} - \frac{1}{w}\frac{dy}{dw} - 4w^2L(w)y = 0.$$
in Eq. \eqref{eq:general} we arrive to differential equation:
\begin{equation}\label{eq:dal}
\frac{d^2y}{dw^2} - \frac{1}{w}\frac{dy}{dw} - 4w^2L(w)y = 0.
\end{equation}
whose associated Riccati equation has a Laurent polynomial solution. We may then transform the equation into 
a reduced form with the same associated Riccati equation by taking, 
$$\tilde y = \frac{y}{\sqrt{w}}$$
obtaining, 
\begin{equation}\label{eq:reducedX}
\frac{d^2\tilde y}{dw^2} = \left(\frac{3}{4w^2} + 4w^2 L(w^2)\right) \tilde y.
\end{equation}

This trace free equation has a Laurent polynomial coefficient whose leading term $w^{2m+2}$ has a coefficient $4$. We scale the independent variable
by taking $\tilde w = \sqrt[m+2]{2} w$ obtaining:
\begin{equation}\label{eq:reduced}
\frac{d^2\tilde y}{d\tilde w^2} = \left(\frac{3}{4 \tilde w^2} + \sqrt[m+2]{2^{2m}} \tilde w^2 L \left( \frac{\tilde w^2}{\sqrt[m+2]{2^2}} \right)\right) \tilde y,
\end{equation}
that can be seen as a differential equation in the family $\mathbb B_{(2,2m+2)}$. \\

\begin{lemma}\label{lm:dal1}
$$L(x)\in \mathbb S_{(2,m)}\quad \Longleftrightarrow \quad \frac{3}{4 x^2} + \sqrt[m+2]{2^{2m}} x^2 L \left( \frac{x^2}{\sqrt[m+2]{2^2}} \right) \in \mathbb B_{(2,2m+2)}$$
\end{lemma}

\begin{proof}
Let part from any $L(x)\in \mathbb M_{(2,m)}$ and let us follow the change of variables and reduction from Eq. \eqref{eq:general} to Eq. \eqref{eq:reduced}. It is clear that if the second equation has a Liouvillian solution, then we obtain a liouvillian solution of \eqref{eq:general}. But, because of the above discussion, if the second equation is Picard-Vessiot integrable, it must be in $\mathbb B_{(2,2m+2)}$.
\end{proof}

Finally let us note that the D'alembert transform and reductions process is a polynomial map in the coefficients of the involved Laurent polynomials. We define:
$$f_{m,r}\colon \mathbb M_{r,m}\hookrightarrow \mathbb M_{2(r-1),2(m+1)}, \quad L(x) \mapsto \frac{3}{4 x^2} + \sqrt[m+2]{2^{2m}} x^2 L \left( \frac{x^2}{\sqrt[m+2]{2^2}} \right)$$
$$(\ell_{-r},\ldots,\ell_{m-1})\mapsto \left(\sqrt[m+2]{2^{2r}}\ell_{-r},0,%\lambda^{-r+1}\ell_{-r+1},0,
\ldots,\frac{3}{4}+\sqrt[m+2]{2^4}\ell_{-2},0,\ldots,\frac{\ell_{m-1}}{\sqrt[m+2]{2^{2(m-1)}}},0\right).$$
It is clearly an affine embedding.

\begin{proposition}
$\mathbb S_{(2,m)} = f^{-1}_{(2,2p)}(\mathbb B_{(2,2m+2)})$.
\end{proposition}

\begin{proof}
It is a direct consequence of Lemma \ref{lm:dal1}.
\end{proof}

\section{Asymptotic iteration method for auxiliary equations}

From Theorems \ref{th:case1}, \ref{th:case2} and \ref{th:case3} we have that in order to characterize the spectral sets ${\mathbb S}_{(1,2p)}$, ${\mathbb B}_{2,2p}$ and $\mathbb S_{(2q,2p)}$ we need to determine in which cases the auxiliary differential equations \eqref{aux1}, \eqref{aux2} and \eqref{aux3} admit a polynomial solution. 
With this purpose we use an adapted version of the asymptotic iteration method or AIM, that  was introduced by H. Ciftci \textit{et al} in \citep{CHS} and previously applied to a related problem in \cite{abv2020}.

\subsection{General considerations about the asymptotic iteration method}

Let us consider $\mathbb Q\{\ad,\bd\}$ the ring of differential polynomials in two differential indeterminates $\ad$, $\bd$. Let us consider a differential equation,
\begin{equation}\label{gen:asymptotic}
y '' = \ad y' + \bd y 
\end{equation}

We also set:
$\lambda_0(\ad,\bd) = \ad, \quad s_0(\ad,\bd) = \bd.$
By derivation of Eq. \eqref{gen:asymptotic} we obtain a sequence of differential equations,
\begin{equation}\label{miaseq}
y^{(j+2)}=\lambda_{j}(\ad,\bd)y'+s_{j}(\ad,\bd)y
\end{equation}
where $\{\lambda_j(\ad,\bd)\}_{j\in\mathbb N}$ and $\{s_j(\ad,\bd)\}_{j\in \mathbb N}$ are sequences of differential polynomials in $\ad$ and $\bd$ defined by the recurrence  (cf. (21) in \cite{abv2020}),
\begin{equation}\label{recurrence_lr}
\lambda_{j+1}=\lambda_{j}'+s_{j}+\ad \lambda_{j}, \quad s_{j+1}=s_{j}'+\bd \lambda_{j}.    
\end{equation}
and a sequence of obstructions,
$$\Delta_j(\ad,\bd) = s_j(\ad,\bd)\lambda_{j-1}(\ad,\bd) - \lambda_j(\ad,\bd)s_{j-1}(\ad,\bd).$$

Another way of looking at this sequence $\{\Delta_n\}_{n\in\mathbb N}$ is to consider the recurrence law \eqref{recurrence_lr} as
the iteration of a $\mathbb Q$-linear operator in $\mathbb Q\{\ad,\bd\}^2$. Note that:
$$\left[
\begin{array}{c} \lambda_{j+1} \\ s_{j+1} \end{array}
\right] = 
\left(
\frac{d}{dx} + 
\left[
\begin{array}{cc} \ad & 1 \\ \bd & 0 \end{array}
\right]
\right)
\left[
\begin{array}{c} \lambda_{j} \\ s_{j} \end{array}
\right]$$
it follows:

\begin{equation}\label{eq:Delta_n}
\Delta_n(\ad,\bd) = {\rm det}\left\{
\left(
\frac{d}{dx} + 
\left[
\begin{array}{cc} \ad & 1 \\ \bd & 0 \end{array}
\right]
\right)^n
\left[
\begin{array}{cc} \ad & 1 \\ \bd & 0 \end{array}
\right]\right\}
\end{equation}

This sequence, $\{\Delta_n(\ad,\bd)\}_{n\in \mathbb N}$ is a sequence of universal differential polynomials in two variables that test  whether a linear second order homogeneous linear differential equation has a polynomial solution. The following is a differential algebraic refinement of a known result ( the original result,  \cite[Theorem 2]{CHS} included some spurious hypothesis and was stated for smooth functions in an open interval, it was later translated into differential algebraic terms \cite[Theorem 4.1]{abv2020} but it was stated in a weaker form than here). In order to state it, let us consider an algebraically closed field $\mathcal C$ of characteristic zero and
a differential field extension $\left(\mathcal K,\partial\right)$ of $\left(\mathcal C(x), \frac{d}{dx}\right)$ whose field of constants is $\mathcal C$. As usual, we write $f'$ for $\partial f$. We consider a differential equation,
\begin{equation}\label{eq:mia2}
y'' = f y' + g y
\end{equation}
where $f$ and $g$ are elements of $\mathcal K$. 

\begin{theorem}\label{ThMiaPoly} 
Differential equation \eqref{eq:mia2} has a polynomial solution in $\mathcal C[x]$ of degree at most $n$ if and only if $\Delta_n(f,g) = 0$. 
\end{theorem}

\begin{proof}
For the first part of the proof we follow the same argument that \cite[Theorem 2]{CHS}.
Let us assume that Eq. \eqref{eq:mia2} admits a polynomial solution $P(x)\neq 0$ of degree at most $n$. Note that from derivation of of Eq. \eqref{eq:mia2} we obtain equations with coefficients in $\mathcal K$,
\begin{equation}\label{eq:casos_n}
\begin{cases}
y^{n+1} &= \lambda_{n-1}(f,g)y' + s_{n-1}(f,g)y\\
y^{n+2} &= \lambda_n(f,g)y' + s_n(f,g)y.
\end{cases}
\end{equation}
By taking $y = P(x)$ we obtain:
$$\left[
\begin{array}{cc} \lambda_{n-1}(f,g) & s_{n-1}(f,g) \\ \lambda_{n}(f,g) & s_{n}(f,g) \end{array}
\right]\left[
\begin{array}{c} P(x)  \\ P'(x) \end{array}
\right] = \left[
\begin{array}{c} 0  \\ 0 \end{array}
\right],$$
if $P(x)\neq 0$ then we have
$$\Delta_n(f,g) = {\rm det}\left[
\begin{array}{cc} \lambda_{n-1}(f,g) & s_{n-1}(f,g) \\ \lambda_{n}(f,g) & s_{n}(f,g) \end{array}
\right] = 0.$$

Now, let us assume $\Delta_n(f,g) = 0$. The key point of the proof is that any homogeneous linear differential equation with coefficients in $\mathcal K$ of order $r$ has a vector space of solutions of dimension $r$ in a suitable extension of $\mathcal K$. We have two different possibilities:
\begin{itemize}
\item[(a)] If $\lambda_{n-1}(f,g)=0$ and $s_{n-1}(f,g) = 0$ then Eq. \eqref{eq:casos_n} yield $y^{n+1} = 0$. This implies that solutions of Eq. \eqref{eq:general} are polynomials of degree at most $n$
\item[(b)] In any other case then there is a unique $h\in \mathcal K$ such that
$h\lambda_{n-1}(f,g) = \lambda_n(f,y)$ and $hs_{n-1}(f,g) = s_n(f,g)$. By substitution in Eq. \eqref{eq:casos_n} we obtain
$$y^{(n+2)} = hy^{n+1}.$$
The general solution of this equation is $(n+2)$-vector space over $\mathcal C$ containing the space of all polynomials of degree at most $n$ as an hyperplane. The space of solutions of Eq. \eqref{eq:mia2} is a $2$-subspace of such $(n+2)$-vector space. It must intersect the hyperplane of polynomials at least along a line. Thus, it contains a polynomial of degree at most $n$.
\end{itemize}

\end{proof}

%{\color{red} Henock: ¿Es posible poner aquí una tabla con los valores de los primeros valores de $Delta_j$. Por sentido común hay que tomar $\Delta_0 = -\beta$ puesto que si $\beta$ es cero entonces la ecuación tiene una solución constante. $\Delta_1 = \alpha\beta' - \alpha'\beta-\beta^2$ supongo que tenemos expresiones razonables hasta el $\Delta_3$ o $\Delta_4$.
%} al final creo que es manejable hasta el $\Delta_3$, el polinomio diferencial $\Delta_4$ es al menos 3 veces más largo que el anterior. 
\begin{table}[ht]
    \centering
    \caption{First universal differential polynomials $\Delta_d$}
    \begin{tabular}{c c}
\hline
    $\Delta_0$ & $-\bd$\\
    $\Delta_1$ & $\ad\bd' - \ad'\bd-\bd^2$ \\
    $\Delta_2$ & $\begin{aligned}&-\ad ''\ad\bd+\ad^2\bd ''-2\ad\ad'\bd'-3\ad\bd'\bd+2(\ad')^2\bd+3\ad'\bd^2+\bd^3-\ad''\bd'+\ad'\bd''\\
    &-2(\bd')^2+\bd\bd''\end{aligned}$\\
    $\Delta_3$&
    $\begin{aligned}
           &3\ad'\ad\bd'''+3\ad'\ad''\bd'-10\ad'\bd''\bd-3\ad\ad''\bd''-6\ad\bd''\bd'+2\ad\bd\bd'''+10\ad''\bd\bd'\\
           &-\ad'''\ad^2\bd-3\ad'\ad^2\bd''-3\ad^2\ad''\bd'-4\ad^2\bd''\bd+4\ad\ad''\bd^2-3(\bd'')^2-\bd^4+\ad^3\bd'''\\
           &-\ad'''\bd^2-3(\ad')^2\bd''+3(\ad'')^2\bd-4\bd''\bd^2-\ad'''\bd''+\ad''\bd'''+2\bd'\bd'''-6(\ad')^3\bd\\
           &-11(\ad')^2\bd^2-6\ad'\bd^3-3\ad^3(\bd')^2+6\ad'(\bd')^2+8\bd(\bd')^2-2\ad'''\ad'\bd-\ad'''\ad\bd'\\
           &+6(\ad')^2\ad\bd'+6\ad\bd^2\bd'+6\ad'\ad\ad''\bd+14\ad'\ad\bd\bd'
    \end{aligned}$
           \end{tabular}
    \label{useries}
\end{table}

\subsection{Asymptotic iteration method for auxiliary equations \eqref{aux1}, \eqref{aux2} and \eqref{aux3} }

Here we complete the integrability analysis of equation \eqref{eq:general} by applying
the asymptotic iteration method to auxiliary equations. 
%Given $L(x) \in \mathbb M_{(1,2p)}$ we consider its decomposition as in equation \eqref{Q_decomp_1}.
We need to introduce the spectral varieties, some subsets of the spectral set that turn out to be algebraic varieties.

\begin{definition}
Let us call $\mathbb S_{(1,2p)}^{(d \pm)}$ to the set of $L(x)\in \mathbb M_{(1,2p)}$ such that the corresponding auxiliary equation \eqref{aux1} with $s_{\infty} = \pm 1$ has a polynomial solution of degree $d$. Let us also set,
$$\mathbb S_{(1,2p)}^{(d)} = \mathbb S_{(1,d)}^{(d+)} \cup \mathbb S_{(1,d)}^{(d-)}.$$
\end{definition}

\begin{definition}
Let us call $\mathbb B_{(2,2p)}^{(d \pm)}$ to the set of $L(x)\in \mathbb M_{(2,2p)}$ such that the corresponding auxiliary equation \eqref{aux2} with $s_{\infty} = \pm 1$ and some choice of $s_0$ has a polynomial solution of degree $d$. Let us also set,
$$\mathbb B_{(2,2p)}^{(d)} = \mathbb B_{(2,p)}^{(d+)} \cup \mathbb B_{(2,2p)}^{(d-)}.$$
\end{definition}

\begin{definition}
Let us set $\mathbb S_{(2,m)}^{(d\pm)} = f_{(2,m)}^{-1}(\mathbb B_{(2,2m+2)}^{(d\pm)})$ and 
$\mathbb S_{(2,m)}^{(d)} = f_{(2,m)}^{-1}(\mathbb B_{(2,2m+2)}^{(d)})$.
\end{definition}

\begin{definition}
Let us call $\tilde{\mathbb S}_{(2q,2p)}^{(d \pm \pm)}$ to the set of $(R(x),B(x),A(x))\in \tilde{\mathbb M}_{(2q,2p)}$ such that the corresponding auxiliary equation \eqref{aux3} with $s_{\infty} = \pm 1$, $s_0 = \pm 1$ has a polynomial solution of degree $d$. Let us also set,
$$\tilde{\mathbb S}_{(2q,2p)}^{(d)} = \tilde{\mathbb S}_{(2q,2p)}^{(d++)} \cup \tilde{\mathbb S}_{(2q,2p)}^{(d+-)} \cup \tilde{\mathbb S}_{(2q,2p)}^{(d-+)} \cup \tilde{\mathbb S}_{(2q,2p)}^{(d--)}.$$
\end{definition}

\begin{definition}
Let us call $\mathbb S_{(2q,2p)}^{(d\pm\pm)} = \pi_{(2q,2p)}^{-1}(\tilde{\mathbb S}^{(d\pm\pm)}_{(2q,2p)} )$ and $\mathbb S_{(2q,2p)}^{(d)} = \pi_{(2q,2p)}^{-1}(\tilde{\mathbb S}^{(d)}_{(2q,2p)})$.
\end{definition}

By the above definitions we have that, for any $(r,m)\in \mathbb Z_+^2$ the spectral set decomposes as enumerable union of strata,
$${\mathbb S}_{(r,m)} = \bigcup_{d\geq 0} {\mathbb S}_{(r,m)}^{(d)}.$$
By application of the asymptotic iteration method we will see that these strata ${\mathbb S}_{(r,m)}^{(d)}$ are algebraic varieties, that we call \emph{spectral varieties}, and compute their equations.

\begin{proposition}\label{th:aim1}
$L(x) \in \mathbb S_{(1,2p)}^{(d \pm)}$ with $s_\infty = \pm 1$ if and only if,
\begin{itemize}
    \item[(a)] $\pm b_{p-1} = 2d + p + 2$.
    \item[(b)] $\Delta_d\left(-s_\infty A(x) - \frac{1}{x}, -s_\infty A'(x)+B(x)+\frac{2s_\infty A(x)-a}{x}  \right)  \in \mathbb C[x,x^{-1}]$ vanishes. 
\end{itemize}
\end{proposition}

\begin{proof}
It follows automatically from Theorem \ref{th:case1} by  application of Theorem \eqref{ThMiaPoly} to \eqref{aux1}.
Note that condition (a) implies that the degree of the polynomial solution of \eqref{aux1} is not strictly smaller than $d$.
\end{proof}

\begin{remark}
Note that given a choice of the sign $s_{\infty}$ then
$$\Delta_d\left(-s_\infty A(x) - \frac{1}{x}, -s_\infty A'(x)+B(x)+\frac{2s_\infty A(x)-a}{x}  \right) \in \mathbb Q[\ell_{-1},\ldots,\ell_{2p},x,x^{-1}].$$
Therefore, the coefficients of $\Delta_d\left(-s_\infty A(x) - \frac{1}{x}, -s_\infty A'(x)+B(x)+s_\infty\frac{2A(x)-a}{x}  \right)$ are regular functions in
$\mathbb M_{(1,2p)}$ and, together with equation $\pm b_{p-1} = 2d + p + 2$, determine the spectral variety 
$\mathbb S_{(1,2p)}^{(d\pm)}$.
\end{remark}

Now, we apply the asymptotic iteration method to Eq. \eqref{aux2} in order to give necessary and sufficient conditions for the existence of a polynomial solution of degree $d$. Given $L(x) \in \mathbb M_{(2,2p)}$ we consider its decomposition as in Eq. \eqref{Q_decomp_2}.

\begin{proposition}\label{th:aim2}
$L(x) \in \mathbb B_{(2,2p)}^{(d \pm)}$, with $s_\infty = \pm 1$, if and only if,
\begin{itemize}
    \item[(a)] $s_\infty b_{p-1} = 2d + s_0\sqrt{1+4b} + p + 1$ (this determines the choice of the square root, that is, the sign $s_0$, and therefore $\lambda$).
    \item[(b)] $\Delta_d\left(-s_\infty A(x) - \frac{\lambda}{x}, -s_\infty A'(x)+B(x)-\frac{2s_\infty\lambda A(x)-a}{x}  \right)  \in \mathbb C[x,x^{-1}]$ vanishes. 
\end{itemize}
\end{proposition}

\begin{proof}
It follows automatically from Theorem \ref{th:case2} by  application of Theorem \ref{ThMiaPoly} to Eq. \eqref{aux2}.
Note that condition (a) implies that the degree of the polynomial solution of Eq. \eqref{aux1} is not strictly smaller than $d$.
\end{proof}

\begin{remark}
Note that given a choice of the sign $s_{\infty}$ then
$$\Delta_d\left(-s_\infty A(x) - \frac{\lambda}{x}, -s_\infty A'(x)+B(x)-\frac{2s_\infty\lambda A(x)-a}{x}  \right) \in \mathbb Q[\ell_{-1},\ldots,\ell_{2p},x,x^{-1}].$$
Therefore, the coefficients of $\Delta_d\left(-s_\infty A(x) - \frac{\lambda}{x}, -s_\infty A'(x)+B(x)-\frac{2s_\infty\lambda A(x)-a}{x}  \right)$ are regular functions in
$\mathbb M_{(1,2p)}$ and, together with equation $$(s_\infty b_{p-1}-2d-p-1)^2 = 1+4b,$$ determine the spectral variety 
$\mathbb S_{(1,2p)}^{(d\pm)}$.
\end{remark}

Finally, we apply the asymptotic iteration method to Eq. \eqref{aux3}. Given $(R(x),B(x),A(x)) \in \tilde{\mathbb M}_{(2q,2p)}$ and a choice of the signs $s_\infty = \pm1$ and $s_0 = \pm 1$ in Eq. \eqref{aux3} let us recall:
$$\lambda = \frac{s_0b_{-(q+1)}}{2r_q} + \frac{q}{2} = \frac{s_\infty b_{p-1} - p}{2} - d,\quad \omega(x) = s_\infty A(x) + s_0R(x).$$

\begin{proposition}\label{th:aim3}
$(R(x),B(x),A(x)) \in \tilde{\mathbb S}_{(2q,2p)}^{(d\pm\pm)}$ if and only if 
\begin{itemize}
    \item[(a)] $d = \frac{s_\infty b_{p-1} -p - q}{2} - \frac{s_0b_{-(q+1)}}{2r_{-q}}.$
    \item[(b)] $\Delta_d\left(-\omega(x) - \frac{\lambda}{x}, 
-\omega'(x) + B(x) - 
\frac{\lambda(2\omega(x) + 
\lambda-1)}{x} \right)\in \mathbb C[x,x^{-1}]$
vanishes.
\end{itemize}
\end{proposition}

\begin{proof}
It follows automatically from Theorem \ref{th:case3} by  application of Theorem \ref{ThMiaPoly} to Eq. \eqref{aux2}.
Note that condition (a) implies that the degree of the polynomial solution of Eq. \eqref{aux1} is not strictly smaller than $d$.
\end{proof}

\begin{remark}
Note that given a choice of the signs $s_0$ and $s_{\infty}$ then
$$\Delta_d\left(\omega(x) + \frac{\lambda}{x}, 
\omega'(x) - B(x) + 
\frac{\lambda(2\omega(x) + 
\lambda-1)}{x} \right) \in \mathbb Q[r_{j}^{-1},b_{j},a_j,r_{-q}^{-1},x,x^{-1}].$$
The coefficients of $\Delta_d(\omega(x),B(x),\lambda)$ as a Laurent polynomial in $x$ are regular functions in
$\tilde{\mathbb M}_{(2q,2p)}$ and its zero locus is the spectral variety $\tilde{\mathbb S}_{(2p,2q)}^{(d\pm\pm)}$.
\end{remark}

\begin{remark}
The projection $\pi_{(2q,2p)}(\tilde{\mathbb S}_{(2q,2p)}^{(d\pm\pm)}) = {\mathbb S}_{(2q,2p)}^{(d\pm\pm)}$ is then an algebraic subvariety of $\mathbb M_{(2p,2q)}$. 
\end{remark}

\subsection{Decomposition of the spectral set}

\begin{theorem}\label{th:main}
For any $(r,m)\in \mathbb Z_{>0}^2$ we have a decomposition of the spectral set
$${\mathbb S}_{(r,m)} = \bigcup_{d\geq 0} {\mathbb S}_{(r,m)}^{(d)}$$
as a enumerable union of spectral varieties. Moreover:
\begin{enumerate}
    \item[(a)] If $L(x)\in \mathbb S_{(r,m)}^{(d)}$ with $r\neq 2$ then the differential equation \eqref{eq:general} has a solution of the form:
$$y(x) = x^{\lambda}P(x)e^{\int \omega(x)dx}$$
where $P(x)$ is a monic polynomial of degree $d$, and $\omega(x)$ is a Laurent polynomial. 
   \item[(b)] If $L(x)\in \mathbb S_{(r,m)}^{(d)}$ with $r=2$ then the differential equation \eqref{eq:general} has a solution of the form:
$$y(x) = x^{\lambda}P(\sqrt{x})e^{\int \omega(\sqrt{x})dx}$$
where $P(x)$ is a monic polynomial of degree $d$, and $\omega(x)$ is a Laurent polynomial.
\end{enumerate}
\end{theorem}

\begin{proof}
Case (a) with $r=1$ and even $m$ follows directly from Proposition \ref{th:aim1}. Case (a) with $r>2$ and even $m$ follows directly from Proposition \ref{th:aim3} and the fact that $\tilde{\mathbb M}_{(r,m)}$ is a $2$-covering space of
$\mathbb M_{(r,m)}$ and the algebraic varieties $\tilde{\mathbb S}_{(r,m)}^{\pm\pm}$ projects onto
algebraic varieties ${\mathbb S}_{(r,m)}^{\pm\pm}$. Finally, case (a) with odd $m$ is trivial as the spectral set is empty. Finally, case $(b)$ for $(2,m)$ follows from \ref{th:aim2} applied to $\mathbb M_{(2,2m+2)}$ and then D'Alembert transform.
\end{proof}

\section{Applications}\label{s:examples}

%In this section we present some examples and applications of our approach regarding seminal works and recent references that include the solution of differential equations with Laurent Polynomial Coefficients. 
First, we consider biconfluent Heun and doubly confluent confluent differential equations, whose solutions are %important in the universe of \emph{special functions}
relevant special functions, see \cite{ron95}. %In particular
Then, we analyze one dimensional stationary Schr\"odinger equations with potentials in $\mathbb{C}[x,x^{-1}]$, $\mathbb{C}[e^{\mu x},e^{-\mu x}]$ and $\mathbb{C}[\sqrt{x},\sqrt{x^{-1}}]$, which are relevant in mathematical physics and that can be written in terms of biconfluent and doubly confluent Heun differential equations. Our approach recovers and extends the results of Duval et. al. in \cite{dulo92} about these equations, which in their normal form are particular cases of the following differential equation:
\begin{equation}\label{eq:grh}
   y''-\left(a_2x^2+a_1x+a_0+\frac{a_{-1}}{x}+\frac{a_{-2}}{x^2}+\cdots \frac{a_{-2k}}{x^{2k}}\right)y=0. 
\end{equation}

\subsection{Biconfluent Heun equation}
Consider the reduced form of biconfluent Heun equation, which is given by
\begin{equation}\label{eq:bhr}
y''-\left(\left(x +\frac{\beta}{2}\right)^2 -\gamma+\frac{\delta}{2x} + \frac{\alpha^2-1}{4x^2}\right)y=0
  \end{equation}
We see that Eq. \eqref{eq:grh} corresponds to Eq. \eqref{eq:bhr} when $k=1$, $a_2=1$, $a_1=\beta$,  $a_0=\frac{\beta^2}{4}-\gamma$, $a_{-1}=\frac{\delta}{2}$ and $a_{-2}=\frac{\alpha^2-1}{4}$. A galoisian analysis of Eq. \eqref{eq:bhr} was developed by Duval et al. in \cite{dulo92}, where the authors obtained  some description of the spectral variety by treating the auxiliary equation by the method of undetermined coefficients. Using our approach we apply Theorem \ref{th:aim3}, after a coefficient decomposition as presented in Eq. \eqref{Q_decomp_3}. %The reduced form of Heun Biconfluent Equation appears in the literatura in the seminal paper of Kovacic (particular examples without the mention of the name of the equation). Thus, Kovacic in 
 %\cite[\S 3.2, pag. 13]{Kov86} analyzed the following differential equations:
%$$y''-\left(x^2-2x+3+\frac{1}{x}+\frac{7}{4x^2}-\frac{5}{x^3}+\frac{1}{x^4}\right)y=0.$$ 

Through our approach for Eq. \eqref{eq:bhr}, the AIM method help us to obtain in an explicit way the vanishing condition of the determinant provided in \cite{dulo92}, which is equivalent to the vanishing of the universal polynomial $\Delta_d$. We provide a criteria for the existence of the solution and we use $d=1,2$ to illustrate the method. Solutions are given in pairs and the Galois group is diagonalizable. That is, we obtain the same universal differential polynomial for cases $s_\infty=1$, $s_0=1$  and $s_\infty=1$, $s_0=-1$ as follows.

For $d=1$:
$$\frac{1}{4}\alpha^{2}\beta^{2}\pm 
\alpha \beta^{2}\mp\frac{1}{2}\alpha\beta\delta\pm 2
\alpha +\frac{3}{4}{\beta}^{2}-\beta\delta+\frac{1}{4}{\delta}
^{2}+2$$

For $d=2$:

$${\frac {9\beta{\delta}^{2}}{8}}-{\frac {23{\beta}^{2}\delta}{8}}
-\frac{3}{8}\alpha^{2}\beta^{2}\delta+4\alpha^{2}\beta\mp\frac{9}{4} 
\alpha\beta^{2}\delta\pm\frac{3}{8}\alpha\beta\delta^{2}\mp 4\alpha \delta$$ $$\pm\frac{1}{8} 
\alpha^{3}\beta^{3}+{\frac {9
\alpha^{2}\beta^{3}}{8}}\pm{\frac {23\alpha\beta^{3}}{8}}-\frac{1}{8}{\delta}^{3}+{\frac {15{\beta}^
{3}}{8}}\pm 18\alpha\beta+14\beta-6\delta$$
In a similar way, we obtain the same universal differential polynomials for cases $s_\infty=-1$, $s_0=-1$  and $s_\infty=-1$, $s_0=1$ as follows.

For $d=1$:
%;
$$\frac{1}{4}\alpha^{2}\beta^{2}\mp \alpha\beta^{2}\mp\frac{1}{2}\alpha\beta\delta\pm 2
\alpha +\frac{3}{4}\beta^{2}+\beta\delta+\frac{1}{4}\delta^{2}-2$$

For $d=2$:
%;
$$-{\frac {9\,\beta\,{\delta}^{2}}{8}}-{\frac {23\,{\beta}^{2}\delta}{8}
}+6\,\delta+14\,\beta-\frac{3}{8}\,  \alpha^{2
}{\beta}^{2}\delta+4\, \alpha  ^{2}
\beta\pm\frac{9}{4}\, \alpha  {\beta}^{2}\delta\pm\frac{3}{8}\, 
\alpha  \beta\,{\delta}^{2}$$
$$\mp 18\,\beta\,  \alpha  \mp
4\, \alpha  \delta-{\frac {15\,{\beta}^{3}}{8}}-\frac{1}{8}\,{
\delta}^{3}\mp\frac{1}{8}\, \alpha ^{3}{\beta}^{
3}-{\frac {9\, \alpha^{2}{\beta}^{3}
}{8}}\pm{\frac {23\, \alpha  {\beta}^{3}}{8}}
$$
\subsubsection{The case $x^2+\frac{\beta}{x}$}
Setting  $\gamma=\frac{\beta^2}{4}$, $\alpha=\pm 1$ and $\delta=0$ in Eq. \eqref{eq:bhr} we have:
\begin{equation}\label{eq:solo_beta}
y''-\left(x^2+\frac{\beta}{x}\right)y=0
\end{equation}
Thus we obtain the the equations of $\mathbb S_{(1,2)}^{d\pm}$ restricted to this one-parametric family in Tables \ref{tb:beta1} and \ref{tb:beta2}.

\begin{table}[ht]
    \centering
    \caption{First universal differential polynomials $\Delta_d$ for $s_\infty=1$ in family \eqref{eq:solo_beta}}\label{tb:beta1}
    \begin{tabular}{c c} 
\hline
    $d$ & $\Delta_d$\\
    1& ${\beta}^{2}+4$\\
    2& $-{\beta}^{3}-20\,\beta$\\
    3& ${\beta}^{4}+60\,{\beta}^{2}+288$\\
    4& $-\beta \left( {\beta}^{4}+140\,{\beta}^{2}+2848 \right) $\\
    5& ${\beta}^{6}+280\,{\beta}^{4}+15280\,{\beta}^{2}+86400$\\
    6& $-\beta \left( {\beta}^{6}+504\,{\beta}^{4}+59184\,{\beta}^{2}+1316736 \right) $\\
    7& ${\beta}^{8}+840\,{\beta}^{6}+185520\,{\beta}^{4}+10460800\,{\beta}^{2}+67737600$\\
    8& $-\beta \left( {\beta}^{8}+1320\,{\beta}^{6}+500016\,{\beta}^{4}+58002560\,{\beta}^{2}+
1416075264 \right)$\\
9& ${\beta}^{10}+1980\,{\beta}^{8}+1202256\,{\beta}^{6}+252901440\,{\beta}^{4}+15116419584
\,{\beta}^{2}+109734912000$\\
10& $-\beta \left( {\beta}^{10}+2860\,{\beta}^{8}+2642640\,{\beta}^{6}+925038400\,{\beta}^{4}+
110532006400\,{\beta}^{2}+2941885440000 \right) $
           \end{tabular}
\end{table}

\begin{table}[ht]
    \centering
    \caption{First universal differential polynomials $\Delta_d$ for $s_\infty=-1$ in family \eqref{eq:solo_beta}}\label{tb:beta2}
    \begin{tabular}{c c}
\hline
    $d$ & $\Delta_d$\\
    1& ${\beta}^{2}-4$\\
    2& $-{\beta}^{3}+44\,\beta$\\
    3& ${\beta}^{4}-180\,{\beta}^{2}+1440$\\
    4& $-\beta\, \left( {\beta}^{4}-500\,{\beta}^{2}+23584 \right)  $\\
    5& ${\beta}^{6}-1120\,{\beta}^{4}+173200\,{\beta}^{2}-1814400$\\
    6& $-\beta\, \left( {\beta}^{6}-2184\,{\beta}^{4}+833328\,{\beta}^{2}-
42947712 \right) 
$\\
    7& ${\beta}^{8}-3864\,{\beta}^{6}+3064368\,{\beta}^{4}-467834752\,{\beta}^
{2}+5811886080
$\\
8&$-\beta\, \left( {\beta}^{8}-6360\,{\beta}^{6}+9338160\,{\beta}^{4}-
3293590400\,{\beta}^{2}+184026470400 \right) 
$\\
9&${\beta}^{10}-9900\,{\beta}^{8}+24753168\,{\beta}^{6}-17322465600\,{
\beta}^{4}+2702445037056\,{\beta}^{2}-38109367296000
$

           \end{tabular}
\end{table}

\subsubsection{A $4$-parametric Hill equation}

Following \cite{CY19} we consider the Hill differential equation
\begin{equation}\label{hill}
    f''(z)+(-e^{4z}+k_3 e^{3z}+k_2 e^{2z}+k_1 e^z +k_0)f(z)=0.
\end{equation}
Through the change of variable  $(z,f(z))\mapsto (x,y)=(e^z,e^{\frac{z}{2}}f(e^z))$ Eq. \eqref{hill} turns into:
%\begin{equation}\label{alghill}
 %   u''(x)+\frac{1}{x}u'(x)+\frac{-x^4+k_3x^3+k_3x^3+k_2x^2+k_1x+k_0}{x^2}u(x)=0,
%\end{equation}
%which is a linear second order differential equation with Laurent coefficients. Moreover, setting $u=x^{-1/2}y$ we obtain the reduced form of \eqref{alghill} throught D'Alembert transform. Thus, we have:
\begin{equation}\label{hillnorm}
y''=\left(\left(x-\frac{k_3}{2}\right)^2-\left(k_2+\frac{k_3^2}{4}\right)-\frac{k_1}{x}-\frac{1/4+k_0}{x^2} \right)y,
\end{equation}
We observe that Eq. \eqref{hillnorm} is a biconfluent Heun equation in the reduced form, Eq. \eqref{eq:bhr} being $\beta=-k_3$, $\gamma=k_2+\frac{k_3^2}{4}$, $\delta=-2k_1$ and $\alpha^2=-4k_0$.

By Theorem \ref{th:case2} we obtain the arithmetic condition $\lambda:=1+s_0 \alpha^2=-(s_\infty\gamma+2d+1)$ over the parameters of Eq. \eqref{hillnorm}. Then, in order to achieve  Picard-Vessiot integrability we should obtain polynomial solutions of the auxiliary equation defined as follows.
\begin{equation}\label{hillaux}
    P_d''(x)+2\left(s_\infty\left(x+\frac{\beta}{2}\right)+\frac{\lambda}{2x}\right)P_d'(x)+\left(s_\infty+\gamma+s_\infty\frac{(x+\frac{\beta}{2})(\lambda-a)}{x}\right)P_d(x)=0
\end{equation}
By the use of Asymptotic Iteration Method to Eq. \eqref{hillaux} we get the conditions for the existence of such solutions. We compute these conditions for $d=1$ and  signs selection $s_\infty=1$, $s_0=1$,
\begin{equation}
  \begin{cases} 
        k_3^6+8k_2 k_3^4+16k_2^2 k_3^2+16k_3^4+64k_2 k_3^2+48k_3^2+256\sqrt{-k_0}+128=0\\
     -k_3^6-8k_2 k_3^4-16k_2^2 k_3^2-16k_3^4-64k_2 k_3^2-16k_3^2+128k_1+128k_2+384=0.
        \end{cases}
\end{equation}
conditions for others small $d$ are easily computable but extremely large.  
\subsubsection{A case with coefficients in $\mathbb{C}[\sqrt{x},\sqrt{x^{-1}}]$.}
In \cite{i2015}, see also \cite{saad2020,l2016} and references therein, was presented the following differential equation
\begin{equation}\label{eq:sqr}
y''-\left(\frac{k_1}{\sqrt{x}}+k_0\right)y=0
  \end{equation}
Setting the change of variables $(x,y)\mapsto (z,u)=(\sqrt{\sqrt{4k_0}x},z^{-1/2}y)$ to transforms Eq. \eqref{eq:sqr} into the following differential equation.
\begin{equation}\label{eq:tsqr}
    u''=\left(\left(z+\frac{J}{2}\right)^2-\frac{J^2}{4}+\frac{3}{4z^2}\right)u,\quad J=\sqrt{2}k_1k_0^{-3/4}.
\end{equation}
We see that Eq. \eqref{eq:grh} corresponds to Eq. \eqref{eq:tsqr} when $k=1$, $a_2=1$, $a_1=J$,   $a_0=a_{-1}=0$ and $a_{-2}=\frac{3}{4}$. We recall that Eq. \eqref{eq:tsqr} corresponds to the reduced form of biconfluent Heun equation \eqref{eq:bhr} with $\alpha=\pm 2$, $\beta=J$, $\gamma=J^2/4$ and $\delta=0$. Applying theorem \ref{th:case2} we find out that a necessary condition to obtain integrability is $\lambda:=\frac{1+2s_0}{2}=\frac{-1}{8}s_\infty J^2-\frac{1}{2}-d$. On the other hand, sufficient condition comes from the polynomial solutions of the auxiliary equation
	\begin{equation}
P''_d+2\bigg(\omega(z)+\frac{\lambda }{z}\bigg)P'_d+ \bigg(\omega'(z)+\frac{J^2}{4}+\frac{2  \lambda \omega(z) }{x}\bigg)P(z)=0
	\end{equation}
where $\omega(z)=s_\infty(z+\frac{J}{2})$. The existence of such polynomials solutions are guaranteed by the vanishing of the obstruction $\Delta_d\left(-2(\omega(z)+\frac{\lambda }{z}),-\omega'(z)-\frac{J^2}{4}-\frac{2  \lambda \omega(z) }{x}\right)$. These obstructions are the same for cases $s_\infty=1$, $s_0=1$ and $s_\infty=1$, $s_0=-1$; $s_\infty=-1$, $s_0=1$ and $s_\infty=-1$, $s_0=-1$ likewise. In tables 4 and 5 we present some computations for smalls values of parameter $d$:
\begin{table}[ht]
    \centering
        \caption{Obstructions $\Delta_d$ for the case $s_\infty=1$ }
        \begin{tabular}{c c}
    \hline
         $d$ & $\Delta_d$ \\
          0  &  ${J}^{2} \left( {J}^{2}+4 \right) $\\
          1  & ${\frac{1}{256}}{J}^{8}+\frac{1}{16}{J}^{6}+\frac{3}{16}{J}^{4}-\frac{1}{2}{J}^{2}-6$\\
          2  &  ${J}^{2} \left( {J}^{10}+36\,{J}^{8}+368\,{J}^{6}+448\,{J}^{4}-15360\,{J}^{2}-102400 \right)$\\
          3  &  $\frac{1}{65536}J^{16}+\frac{1}{1024}J^{14}+\frac{43}{2048}J^{12}+\frac{39}{256}J^{10}-\frac{175}{256}J^8-\frac{299}{16}J^6-\frac{369}{4}J^4+108J^2+1260$
    \end{tabular}
\end{table}

\begin{table}[ht]
    \centering
        \caption{Obstructions $\Delta_d$ for the case $s_\infty=-1$ }
        \begin{tabular}{c c}
    \hline
         $d$ & $\Delta_d$ \\
          0  &  ${J}^{2} \left( {J}^{2}-4 \right) $\\
          1  & ${\frac {1}{256}}{J}^{8}-\frac{1}{16}{J}^{6}+\frac{3}{16}{J}^{4}-\frac{1}{2}{J}^{2}+6$\\
          2  &  ${J}^{2} \left( {J}^{10}-36\,{J}^{8}+368\,{J}^{6}-1472\,{J}^{4}+15360\,{J}^{2}-102400 \right)$\\
          3  &  $\frac{1}{65536}J^{16}-\frac{1}{1024}J^{14}+\frac{43}{2048}J^{12}-\frac{49}{256}J^{10}+\frac{385}{256}J^8-\frac{299}{16}J^6+\frac{387}{4}J^4-108J^2+1260$
    \end{tabular}
\end{table}
\subsection{Doubly confluent Heun equation}
Following Duval et. al. (see \cite[pp. 236]{dulo92}) we consider the reduced form of doubly confluent Heun equation given by
\begin{equation}\label{eq:dhr}
y''-\left(\frac{\alpha^2}{4} -\frac{\gamma}{x}-\frac{\delta}{x^2} -\frac{\beta}{x^3} +\frac{\alpha^2}{4x^4}\right)y=0
  \end{equation}
We see that Eq. \eqref{eq:grh} corresponds to Eq. \eqref{eq:dhr} when $k=2$, $a_2=a_1=0$,   $a_0=\frac{\alpha^2}{4}$, $a_{-1}=-\gamma$, $a_{-2}=-\delta$, $a_{-3}=-\beta$ and $a_{-4}=\frac{\alpha^2}{4}$.
Instead of the using the galoisian approach given by Duval et. al in \cite{dulo92}, we set the change of variables $(x,y)\mapsto (z,u)=(\sqrt{\frac{\sqrt{\alpha}}{x}},z^{3/2}y)$ to transforms Eq. \eqref{eq:dhr} into the following differential equation.
\begin{equation}\label{eq:dht}
    u''=\left(z^2-\frac{4\beta}{\widetilde{\alpha}}-\frac{4\delta-\frac{3}{4}}{z^2}-
    \frac{4\widetilde{\alpha}\gamma}{z^4}+\frac{\widetilde{\alpha}^6}{z^6}\right)u,\quad \widetilde{\alpha}=\sqrt{\alpha}.
\end{equation}
Now, we can use our approach to apply Theorem \ref{th:aim3} to Eq. \eqref{eq:dht}, which corresponds to Eq. \eqref{eq:grh} being $k=3$, $a_2=1$, $a_1=a_{-1}=a_{-3}=a_{-5}=0$, $a_0=-4\beta/\widetilde{\alpha}$, $a_{-2}=-(4\delta-3/4)$, $a_{-4}=-4\widetilde{\alpha}\gamma$ and $a_{-6}=\widetilde{\alpha}^6$.

By theorem \ref{th:aim3} we obtain a necessary condition $\lambda:=-s_02\frac{\gamma}{\widetilde{\alpha}^2}+\frac{3}{2}=\frac{-s_\infty4\beta-\widetilde{\alpha}}{2\widetilde{\alpha}}-d$ over the parameters of equation \eqref{eq:dht}. Then, by means of asymptotic iteration method we search for polynomial solutions of the auxiliary equation
\begin{equation}
    P_d''+2\left(\omega(z)+\frac{\lambda}{z}\right)P_d'+\left(\omega'(z) - B(z) + \frac{\lambda(2\omega(z) + \lambda-1)}{z}\right)P_d=0
\end{equation}
where $\omega(z)=s_\infty z+s_0\frac{\widetilde{\alpha}^3}{z^3}$ and $B(z)=-\frac{4\beta}{\widetilde{\alpha}}-\frac{4\delta-\frac{3}{4}}{z^2}-
    \frac{4\widetilde{\alpha}\gamma}{z^4}$. In order to illustrate this process we compute the polynomial system which solutions are the suitable selection for $\alpha$, $\beta$, $\gamma$ and $\delta$ that vanishes $\Delta_1\left(-2(\omega(z)+\frac{\lambda}{z}),-\omega'(z) + B(z) -\frac{\lambda(2\omega(z) + \lambda-1)}{z}\right)$ for each combination of signs $s_\infty$ and $s_0$:

for $s_\infty=1$ and $s_0=1$ we obtain the equation of the intersection of $\mathbb S^{(1++)}_{(6,2)}$ with the family \eqref{eq:dht}
\begin{equation}
\begin{cases}
256\,{\delta}^{2}+224\,\delta+33=0\\
16\,\beta\,\delta+15\,\widetilde\alpha+23\,\beta=0\\
16\,\widetilde\alpha\,\delta-9\,\widetilde\alpha-16\,\beta=0\\
16\,{\beta}^{2}-272\,\delta\,\gamma-87\,\gamma=0\\
\widetilde\alpha\,\beta+16\,\delta\,\gamma+6\,\gamma=0\\
{\widetilde\alpha}^{2}-16\,\delta\,\gamma-7\,\gamma=0
\end{cases}
\end{equation}
the solution of this system is the union of four algebraic curves: 

\begin{equation}
\left\{ \widetilde\alpha^2=4\gamma,\beta^2=\frac{9}{16}\widetilde\alpha^2 ,\delta=-\frac{3}{16} \right\},\\
\left\{ \widetilde\alpha^2=-4\gamma,\beta^2=-\frac{25}{16}\widetilde\alpha^2 ,\delta=-\frac{11}{16} \right\}    
\end{equation}

for $s_\infty=1$ and $s_0=-1$ we obtain the equation of the intersection of $\mathbb S^{(1+-)}_{(6,2)}$ with family \eqref{eq:dht}
\begin{equation}
    \begin{cases}
256\,{\delta}^{2}+224\,\delta+33=0\\
16\,\beta\,\delta+15\,\widetilde\alpha+23\,\beta=0\\
16\,\widetilde\alpha\,\delta-9\,\widetilde\alpha-16\,\beta=0\\
16\,{\beta}^{2}+272\,\delta\,\gamma+87\,\gamma=0\\
\widetilde\alpha\,\beta-16\,\delta\,\gamma-6\,\gamma=0\\
{\widetilde\alpha}^{2}+16\,\delta\,\gamma+7\,\gamma=0

 \end{cases}
\end{equation}

the solution of this system is the union of four algebraic curves: 

\begin{equation}
\left\{ \widetilde\alpha^2=-4\gamma,\beta^2=-\frac{9}{16}\widetilde\alpha^2 ,\delta=-\frac{3}{16} \right\},\\
\left\{ \widetilde\alpha^2=4\gamma,\beta^2=\frac{25}{16}\widetilde\alpha^2 ,\delta=-\frac{11}{16} \right\}    
\end{equation}

Further computation shows that for family Eq. \eqref{eq:bhr} the intersections with $\mathbb S^{(1++)}_{(6,2)}$ and with $\mathbb S^{(1-+)}_{(6,2)}$ coincide. The same happens for the intersections with $\mathbb S^{(1+-)}_{(6,2)}$ and $\mathbb S^{(1--)}_{(6,2)}$. This implies that two of the auxiliary equations have polynomial solutions simultaneously, there are two linearly independent solutions that are eigenvectors of the Galois group which is then diagonalizable.

\subsection{Perturbed canonical equation}

Another interesting example is the differential equation
\begin{equation}\label{eqcan}
    y''=\left(x^n +\mu x^{n-1}+\frac{m(m+1)}{x^2}\right)y,
\end{equation}
which was deeply analyzed in \cite{abv2020} for $m=0$. Now, by application of theorem \ref{th:case2} we can determine an arithmetic condition on the parameters that needs to be fulfilled, that is 
\begin{equation}
    \mu=s_\infty(2d+n+1+s_0(2m+1)).
\end{equation}
This condition implies that we have four different cases to consider related to the selection of signs $s_0$ and $s_\infty$. Let us take for example case $s_0=s_\infty=1$, so we have the following auxiliary equation
 \begin{equation}\label{canpert}
     P_d''+2\left(x^n+\frac{m+1}{x}\right)P_d'-2dx^{n-1}P_d=0.
 \end{equation}
Through the change of variables $L(z)=P_d(x)$ and $z=-\frac{2x^{n+1}}{n+1}$, the equation \eqref{canpert} turns into:
\begin{equation}
    zL''(z)+\left(\frac{2m+1}{n+1} +1-z\right)L'(z)+\frac{d}{n+1}L(z)=0.
\end{equation}
This last equation is a generalized Laguerre equation and its solutions are confluent hypergeometric function of the first kind $_1F_1\left(-\frac{d}{n+1},\frac{2m+1}{n+1}+1,-\frac{2x^{n+1}}{n+1}\right)$. Thus, equation \eqref{eqcan} will be Picard-Vessiot integrable if $d\equiv 0$ mod $(n+1)$ and a solution is given by
\begin{equation}
   y_{n,d}(x)= x^{m+1}L_{\frac{d}{n+1}}^{\frac{2m+1}{n+1}}\scriptstyle\left(-\frac{2x^{n+1}}{n+1}\right)e^{\frac{x^{n+1}}{n+1}}.
\end{equation}
In general, the change of independent variable $z=-s_\infty\frac{2x^{n+1}}{n+1}$ transform the corresponding auxiliary equation into a generalized Laguerre equation. In the table 6 we summarize some relations of the parameters in equation \eqref{eqcan} that must be satisfied in order to obtain integrability for each selection of signs.

\begin{table}[ht]\label{table:canonical}
    \centering
     \caption{Relations between parameters in equation \eqref{eqcan}, solutions $y_{n,d}$.}
    \begin{tabular}{c c c c c}
       case  & $\mu$ & $z$  & $n\sim d$&$y_{n,d}$\\
       \hline
       $s_\infty=s_0=1$ & $2d+2m+2+n$ & $z=-\frac{2x^{n+1}}{n+1}$ & $n+1|d$& $x^{m+1}L_{\frac{d}{n+1}}^{\frac{2m+1}{n+1}}\scriptstyle\left(-\frac{2x^{n+1}}{n+1}\right)e^{\frac{x^{n+1}}{n+1}}$\\
       $s_\infty=1,\hspace{0.5pt} s_0=-1$& $2d-2m+n$   &$z=-\frac{2x^{n+1}}{n+1}$ & $n+1|d$ & $x^{-m}L_{\frac{d}{n+1}}^{-\frac{2m+1}{n+1}}\scriptstyle\left(-\frac{2x^{n+1}}{n+1}\right)e^{\frac{x^{n+1}}{n+1}}$\\
       $s_\infty=s_0=-1$& $-(2d-2m+n)$ & $z=\frac{2x^{n+1}}{n+1}$& $n+1|d$& $x^{-m}L_{\frac{d}{n+1}}^{-\frac{2m+1}{n+1}}\scriptstyle\left(\frac{2x^{n+1}}{n+1}\right)e^{-\frac{x^{n+1}}{n+1}}$\\
       $s_\infty=-1,\hspace{0.5pt} s_0=1$& $-(2d+2m+2+n)$ &$z=\frac{2x^{n+1}}{n+1}$&$n+1|d$&$x^{m+1}L_{\frac{d}{n+1}}^{\frac{2m+1}{n+1}}\scriptstyle\left(\frac{2x^{n+1}}{n+1}\right)e^{-\frac{x^{n+1}}{n+1}}$
    \end{tabular}
\end{table}

%\subsection{Heun Double Confluent Differential Equation}

%Kovacic in \cite[\S 3.2, pag. 14]{Kov86} and in cite[\S 4.2, pag. 19]{Kov86}  analyzed the following differential equations:
%$$y''+\left(1-\frac{4n^2-1}{4x^2}\right)y=0.$$  
%$$y''-\left(\frac{1}{x}-\frac{3}{16x^2}\right)y=0.$$ 

\subsection{Algebraically solvable potentials for Schr\"odinger equation}

%The results of this paper are very useful for researchers interested in 

There is some interest in understanding Liouvillian solutions for Schr\"odinger equations with Mie type potentials that come from Supersymmetric Quantum Mechanics, Atomic and Molecular Physics, among others, see \cite{AAT18,BD96,CY19,GV09,saad2020,T88} among others. Algebraically solvable potentials correspond to those that have an infinite discrete spectrum and Liouvillian eigenfunctions. 
The following is an extension of Corollary 5.3 in \cite{AMW11}, see also Corollary 2.2.3 in \cite{A09,A10}.
\begin{corollary}\label{cor1}
Assume that $V(x)\in\mathbb{M}_{(r,m)}$ is an algebraically solvable potential with $m>0$. Then, $m=2$ and $r\in\{0,1,2\}$.
\end{corollary}
\begin{proof} It follows directly that: 
\begin{itemize}
    \item $(r,m)=(2,1)$ by applying of part (a) in Theorem \ref{th:aim1}, or
    \item $(r,m)=(2,2)$ by applying of Theorem \ref{th:aim2}, part (a) due to the relation with $E$ or part (b) due to it is valid $\forall \ell \in \mathbb{Z}$, or
    \item $(r,m)=(2,0)$ by the applying of part (a) due to the relation with $E$ and part (b) in Theorem \ref{th:aim3}. Thus, we obtain that $q=0$, otherwise we will have quasi-solvability in the potential (finite number of points in the algebraic curve in the spectral variety).
\end{itemize}
Thus, we conclude the proof.
\end{proof}

\begin{remark}
We can notice that Eq. \eqref{eq:grh} not only include the reduced forms of biconfluent and doubly confluent Heun equations, also include some Schr\"odinger Equations with $3D$ algebraically solvable potentials and energy levels denoted by $E$ such as follows:
\begin{itemize}
    \item Harmonic Oscillator, see \citep[\S 5.2, pp. 331--333]{AMW11}, being $k=1$, $a_2=1$, $a_1=b_1=0$, $a_0=-(2\ell+3+E)$ and $b_2=\ell(\ell +1)$.
    \item Coulomb, see \citep[\S 5.2, pp. 333--336]{AMW11}, being $k=1$, $a_2=a_1=0$, $a_0=1-E$,  $b_1=-2(\ell+1)$ and $b_2=\ell(\ell +1)$.
    \item Algebraic form of Morse, see \citep[\S 6.3, pp. 355--358]{AMW11}, being $k=1$, $a_2=a_1=0$, $a_0=1-E$,  $b_1=-2(\ell+1)$ and $b_2=\ell(\ell +1)$.
\end{itemize}
Concerning Eq. \eqref{hill} and Eq. \eqref{eq:sqr} we observe that they are Schr\"odinger equations with algebraically solvable potentials in where $E=-k_0$ and their algebraic form, Eq. \eqref{hillnorm} and Eq. \eqref{eq:tsqr}, fall in Eq. \eqref{eq:grh}. The potential associated to Eq. \eqref{eq:sqr}, $K_1/\sqrt{x}$, is well known as \emph{inverse square root potential}  where $E=(k_1/\sqrt{2}n)^{4/3}$ %which is very important 
in quantum mechanics. On the other hand, the potential associated to the Hill equation \eqref{hill}, $-e^{4z}+k_3 e^{3z}+k_2 e^{2z}+k_1 e^z$, with $E=-k_0=\left(s_\infty(n+1)+\frac{1}{2}\left(k_2+\frac{k_3^2}{4}\right) \right)^2$, is relevant to %very important in 
complex oscillation theory.\\

%\textcolor{red}{Ojo, en la literatura esta es la energia para el Hill: $E=-k_0=\left(\frac{n+1}{2}+\frac{s_\infty}{4}\left(k_2+\frac{k_3^4}{4}\right)\right)^2$}
\end{remark}

\section*{Acknowledgements}

The authors acknowledge the support of Universidad Nacional de Colombia. The initial stages of this research were funded by Colciencias project ``Estructuras lineales en geometr\'ia y topolog\'ia'' 776-2017 code 57708 (Hermes UN 38300). P.A-H was supported in the final stage by FONDOCYT project  ``Localizaci\'on de valores propios de H-Matrices'' code 2016-2017-057. The computations in this paper were performed by using Maple\textsuperscript{TM}. We also acknowledge the contribution of the anonymous refereees who make important suggestions that helped to improve the quality of this paper.

\appendix

\section{Some elements of Kovacic Algorithm}\label{ap:kov}

In this appendix we review the elements of Kovacic's algorithm that we use in the proofs. Since we are not making use of the complete algorithm we only review the parts that are relevant for the results exposed in this paper. We include the steps, and some explanations, but not the proofs. The interested reader may consult J. Kovacic's original paper \cite{Kov86}; we follow the same notation.

This algorithm assesses the Picard-Vessiot integrability and find Liouvillian solutions, if they exist, of the linear
differential equation
\begin{equation}\label{eq:kov_lin}
\xi''=r\xi,\quad r\in\mathbb{C}(x)
\end{equation}

Picard-Vessiot integrability of equation \eqref{eq:kov_lin} is governed by its associated Riccati equation 
\begin{equation}\label{eq:kov_ric}
\theta^{\prime}=r-\theta ^{2}=\left( \sqrt{r}-\theta\right)
\left(  \sqrt{r}+\theta\right),\quad\theta={\xi'\over \xi}.
\end{equation}
If $u$ is a solution of \eqref{eq:kov_ric} then $\xi = \rm{exp}\left( \int u \right)$
is a solution of \eqref{eq:kov_lin}. According to the properties of the associated Riccati equation we obtain the 4 cases of the algorithm:
\begin{enumerate}
    \item Case 1. The Riccati equation has a rational solution $\theta \in \mathbb C(x)$: the Galois group of \eqref{eq:kov_lin} is conjugated to a subgroup of the Borel subgroup ${\rm B}_2$. The algorithm returns a Liouvillian solution that is an eigenvalue of the Galois group.
    \item Case 2. The Riccati equation has an algebraic solution of degree $2$ over $\mathbb C(x)$: the Galois group of \eqref{eq:kov_lin} is conjugated to a subgroup of the the infinite dihedral group ${\rm D^{\infty}}$ not contained in the identity component of ${\rm D^{\infty}}$. The algorithm returns a pair of Liouvillian solutions whose logarithmic derivatives are conjugated by the action of the Galois group. 
    \item Case 3. The Riccati equation has an algebraic solution of degree $>2$ over $\mathbb C(x)$:  the Galois group of \eqref{eq:kov_lin} is the group of Tetrahedral, Octahedral or Icosahedral symmetries. The algorithm returns the minimal polynomial of the algebraic solutions.
    \item Case 4. The Riccati equation has no algebraic solutions:the Galois group of \eqref{eq:kov_lin} is ${\rm SL}_2(\mathbb C)$. There is no Liouvillian solutions.
\end{enumerate}

The application of the algorithm proceed in the following way. In case 1 we examine exhaustively all the possible candidates for rational solution of \eqref{eq:kov_ric}. If we find one, then we produce a Liouvillian solution of \eqref{eq:kov_lin} and the algorithm stops. If not, we will go on to case 2. In case 2 we examine exhaustively all the possible candidates for algebraic solutions of degree $2$ of \eqref{eq:kov_ric}. As above, if we succeed we produce a solution of \eqref{eq:kov_lin} and the algorithm stops. If not we proceed to case 3. In step three we use the particular knowledge about the remaining candidates for the Galois group and their semi-invariants. If we succeed then we produce algebraic solutions of \eqref{eq:kov_lin}. Else, we proceed to case 4 an claim the non-integrability of the equation.

Denote by $\Gamma$ the set of finite poles of $r$, and define $\Gamma' = \Gamma \cup \{\infty\}$. By the order of $r$ at  $c\in \Gamma$ we mean the order of $c$ as a pole of $r$, that we denote $\circ (r_c)$. By the order of $r$ at $\infty$, $\circ\left(
r_{\infty}\right)$, we mean the order of $\infty$ as a zero of
$r$.\footnote{In the case of our results with $L(x)\in \mathbb C[x,x^1]$ whe have
$\circ(L(x)_0) = r$ and $\circ(L(x)_{\infty})= -m$.} 

There are some necessary conditions on the orders of $r$ at its poles and $\infty$ that allow us to discard any of the cases before carrying out the steps of the computation. In the following theorem we summarized the necessary conditions that are relevant for our proof, all of them shown in \cite[p. 8]{Kov86}.  

\begin{theorem}\label{ap:NC}
The following conditions are necessary for the respective cases to hold.
\begin{itemize}
    \item[Case 1.] Every pole of $r$ must have even order or else have order 1. The order of $r$ at $\infty$ must be even or else be greater than $2$.
    \item[Case 2.] $r$ must have at least one pole that either has odd order greater than $2$ or else has order $2$.
    \item[Case 3.] The order of a pole $r$ cannot exceed $2$ and the order of $r$ at $\infty$ must be at least $2$.
\end{itemize}
\end{theorem}

Let us continue with the description of some relevant steps of the algorithm, that we apply only in case the necessary conditions are satisfied.

\subsection{Case 1. Step 1.} 
For each $p\in \Gamma'$ we will define a pair of complex numbers $\alpha^\pm_p$ and a rational function $\left[ \sqrt{r}\right] _{\infty}$. 

For each $c \in \Gamma$ we define the complex numbers $\alpha_c^{\pm}$ and the rational function $[\sqrt{r}]_c$ as follows:
\begin{description}
\item[$(c_{1})$] If $\circ\left(  r_{c}\right)  =1$, then
$$\left[ \sqrt {r}\right] _{c}=0,\quad\alpha_{c}^{\pm}=1.$$

\item[$(c_{2})$] If $\circ\left(  r_{c}\right)  =2,$ and $$r= b(x-c)^{-2}+\cdots,\quad \text{then}$$
$$\left[ \sqrt {r}\right]_{c}=0,\quad \alpha_{c}^{\pm}=\frac{1\pm\sqrt{1+4b}}{2}.$$

\item[$(c_{3})$] If $\circ\left(  r_{c}\right)  =2\nu\geq4$, 
then $[\sqrt{r}]_c$ is defined as the part of the Laurent series of $\sqrt{r}$ that spans from $(x-c)^{-\nu}$ to $(x-c)^{-2}$. Following \cite[p. 9]{Kov86} we do not compute the Laurent series of $\sqrt{r}$ explicitly but rather we determine it by using undetermined coefficients using the fact that the difference:
$[\sqrt{r}]_c^2 - r$
is a rational function with a pole of order $\leq \nu+1$. We have then:
$$r=
(a_\nu\left( x-c\right)  ^{-\nu}+...+a_2\left( x-c\right)
^{-2})^{2}+b(x-c)^{-(\nu+1)}+\cdots,\quad \text{then}$$ $$\left[
\sqrt {r}\right] _{c}=a_v\left( x-c\right) ^{-v}+...+a_2\left(
x-c\right) ^{-2},\quad\alpha_{c}^{\pm}=\frac{1}{2}\left(
\pm\frac{b}{a_v}+\nu\right).$$
\end{description}
 
 We define $\alpha_\infty^{\pm}$: and $[\sqrt{r}]_\infty$ as follows:
\begin{description}
\item[$(\infty_{1})$] If $\circ\left(  r_{\infty}\right)  >2$, then
$$\left[\sqrt{r}\right]  _{\infty}=0,\quad\alpha_{\infty}^{+}=0,\quad\alpha_{\infty}^{-}=1.$$

\item[$(\infty_{2})$] If $\circ\left(  r_{\infty}\right)  =2,$ and
$r= \cdots + bx^{2}+\cdots$, then $$\left[
\sqrt{r}\right]  _{\infty}=0,\quad\alpha_{\infty}^{\pm}=\frac{1\pm\sqrt{1+4b}%
}{2}.$$

\item[$(\infty_{3})$] If $\circ\left(  r_{\infty}\right) =-2\nu\leq0$
then then $[\sqrt{r}]_\infty$ is defined as the polynomial of degree $d$ that corresponds to the polynomial part of the Laurent series of $\sqrt{r}$ that at $\infty$. As before we do not compute the Laurent series of $\sqrt{r}$ explicitly but rather we determine it by using undetermined coefficients and the fact that the difference:
$[\sqrt{r}]_\infty^2 - r$
should have a polynomial part of degree $\leq v$, therefore:
$$r=\left( a_\nu x^{\nu}+...+a_0\right)^{2}+ bx^{\nu-1}+\cdots,\quad \text{then}$$
$$\left[  \sqrt{r}\right]  _{\infty}=a_\nu x^{v}+...+a_0,\quad
\text{and}\quad \alpha_{\infty}^{\pm }=\frac{1}{2}\left(
\pm\frac{b}{a_\nu}-\nu\right).$$
\end{description}
\medskip

\subsection{Case 1. Step 2.} 
Let us consider all possible assignations $s\colon \Gamma' \to \{+,-\}$. So that any assignation correspond to a choice of sign $s(p)$ for each pole $p\in \Gamma'$. For each assignation $s$ we obtain a complex number:
$$d_s = \alpha_{\infty}^{s
(\infty)}-%
{\displaystyle\sum\limits_{c\in\Gamma}
\alpha_{c}^{s(c)}}$$
Each assignation $s$ giving rise to a non-negative integer number $d_s$ requires an iteration of step $3$. If none of this numbers is a non-negative integer, we discard case 1, and move to case 2.

\subsection{Case 1. Step 3.}. The input for this step is a sign assignation $s$ giving rise to a non-negative integer $d= d_s$. We define the rational function.
$$\omega= s\left(
\infty\right)  \left[  \sqrt{r}\right]  _{\infty}+%
{\displaystyle\sum\limits_{c\in\Gamma^{\prime}}}
\left(s\left(  c\right)  \left[  \sqrt{r}\right]  _{c}%
+{\alpha_{c}^{s(c)}}{(x-c)^{-1}}\right).$$
We look for a monic polynomial
$P_m$ of degree $m$ with
$$P_m'' + 2\omega P_m' + (\omega' + \omega^2 - r) P_m = 0.$$
then then $\xi_1=P_m e^{\int\omega}$ is a
solution of the differential equation \eqref{eq:kov_lin} and an eigenvector of its Galois grou´p. If we run this step for all the assignations $s$ with non-negative integer $m_s$ finding no polynomial solution, then case $1$ does not hold.

\subsection{Case 2. Step 1} 
For each $c\in\Gamma$ we define the 
$E_{c}\subset\mathbb{Z}$ as
follows.
\medskip

\begin{description}
\item[($c_1$)] If $\circ\left(  r_{c}\right)=1$, then $E_{c}=\{4\}$

\item[($c_2$)] If $\circ\left(  r_{c}\right)  =2,$ and $r= \cdots +
b(x-c)^{-2}+\cdots ,\ $ then $$E_{c}=\left\{
2+k\sqrt{1+4b}:k=0,\pm2\right\}.$$

\item[($c_3$)] If $\circ\left(  r_{c}\right)  =\nu>2$, then $E_{c}=\{\nu\}$

We also define the set $E_{\infty}$ as follows:

\begin{description}
\item[$(\infty_{1})$] If $\circ\left(  r_{\infty}\right)  >2$, then
$E_{\infty }=\{0,2,4\}$
\end{description}

\item[$(\infty_{2})$] If $\circ\left(  r_{\infty}\right)  =2,$ and
$r= \cdots + bx^{-2}$, then $$E_{\infty }=\left\{
2+k\sqrt{1+4b}:k=0,\pm2\right\}.$$

\item[$(\infty_{3})$] If $\circ\left(  r_{\infty}\right)  =\nu<2$,
then $E_{\infty }=\{\nu\}$
\medskip
\end{description}

\subsection{Case 2. Step 2.} 
Let us consider all the assignations,
$$e \colon \Gamma \to \bigcup_{p\in \Gamma} E_p$$
such that for all $p\in \Gamma$ we have $e(p)\in E_p$. For each assignation we have an associated complex number:
$$d_e = \frac{1}{2}\left(  e_{\infty}-
{\displaystyle\sum\limits_{c\in\Upsilon^{\prime}}} e_{c}\right)$$
Only the assignations $e$ giving rise to a non-negative integer $d_e$ should be retained and pass to step $3$. If none of this numbers is a non-negative integer then we discard case 2.

\subsection{From Case 2. Step 2. on.}
The rest of the algorithm are not used in this paper. We send the interested reader to the original paper \cite{Kov86}.

\section{Algorithmic computation of spectral varieties}\label{ap:algaim}
Let us consider a family of second order differential equations of the form
$$\frac{d^2 y}{dx}=L(x)y$$
where $L(x)$ is a monic Laurent polynomial whose coefficients are polynomials in some variables $a_1,\ldots,a_k$. We also assume that the coefficient of lower degree of $L$ does not vanish. Now, it follows the description of the algorithm. 

{ \tt

\begin{itemize}
    \item[Input:] $L$, $n$. 
    \item[Output:] The polynomial equations of the variety $\mathbb{S}_n$ in the variables $a_1, \ldots a_k$. \newline
\end{itemize}

\begin{description}
\item[\bf Step 1] Determine if the type of the Laurent polynomial is $(1,2q)$, $(2,2q)$, $(2p,2q)$ or neither of them.

\item[\bf Step 2] Move to Case 1, Case 2, Case 3, or Case 4 accordingly.

\begin{description}
\item[\textbf{Case 1, $(1,2q)$}]:

\begin{description}
   \item[(i)] Decompose the Laurent polynomial as in Eq. \eqref{Q_decomp_1}.
   \item[(ii)]  Determine the auxiliary equation for each choice of sign $s_\infty=\pm 1$.
   \item[(iii)] Compute the universal differential polynomial $\Delta_n$ for each equation in the last step. 
   \item[(iv)] Compute a system of generators of the ideal generated by the coefficients of $\Delta_n$.
    \item[Output:] A set of polynomial equations of two components of $\mathbb{S}_d$ in the variables $a_1, \ldots a_k$.
\end{description}

\item[\textbf{Case 2, $(2,2q)$}]: 

 \begin{description}
     \item[(i)] Apply a change of variable $x = w^2$ obtaining a new differential equation \eqref{eq:reducedX} with new Laurent polynomial $\tilde L(x)$.
     \item[(ii)] Decompose the Laurent polynomial $\tilde L(x)$ as in Eq. \eqref{Q_decomp_2}.
     \item[(iii)] For each choice of sign $s_\infty = \pm 1$ determine the corresponding auxiliary equations \eqref{aux2}.
     \item[(iv)] Compute the universal differential polynomial $\Delta_n$ for each equation in the last step. 
     \item[(v)] Compute a system of generators of the ideal generated by the coefficients of $\Delta_n$.
     \item[Output] Polynomial equations of the two $\mathbb{S}_d$ in the variables $a_1, \ldots a_k$.
 \end{description}

\item[\textbf{Case 3, $(2p,2q)$}]: 

\begin{description}
    \item[(i)] Decompose the Laurent polynomial as in Eq. \eqref{Q_decomp_3}
    \item[(ii)] Determine the auxiliary equation for each choice of sign $s_0=\pm 1$, $s_\infty=\pm 1$.
    \item[(iii)] Compute the universal differential polynomial $\Delta_n$ for each equation in the last step. 
    \item[(iv)] Compute a system of generators of the ideal generated by the coefficients of $\Delta_n$.
    \item[Output:] Polynomial equations of the four $\mathbb{S}_d$ in the variables $a_1, \ldots a_k$.
\end{description}

\item[\textbf{Case 4}]:

\begin{description}
    \item[Output:]  $\mathbb S_d = \emptyset.$ 
\end{description}

\end{description}

\end{description}
}

All parts of the algorithm are implementable in a symbolic computation system. We have implemented successfully in \texttt{Maple}\textsuperscript{\texttrademark} some parts that we used in the applications and AIM sections.

\bibliography{mybibfile}
\end{document}